\documentclass{article}[12pt]

\usepackage{pstricks}
\usepackage{graphics}
\usepackage{caption}
\usepackage{amsmath}
\usepackage{amsthm}
\usepackage{amsfonts}
\usepackage{amssymb}
\usepackage{amscd}
\usepackage{calc}
\usepackage{graphicx}
\usepackage[arc, all]{xy}
\usepackage{color}
\usepackage{wasysym}
\usepackage{setspace}
\usepackage{subcaption}
\usepackage{enumerate}

\newtheorem{thm}{Theorem}[section]

\newtheorem{lemma}[thm]{Lemma}
\newtheorem{prop}[thm]{Proposition}
\newtheorem{cor}[thm]{Corollary}
\newtheorem{conj}{Conjecture}

\newcommand{\Nb}{\mathbb{N}}
\newcommand{\Zb}{\mathbb{Z}}

\newcommand{\LL}{\mathcal{L}}

\newcommand{\CP}{\mathcal{CP}}

\newcommand{\II}{\mathcal{I}}
\newcommand{\ZZ}{\mathcal{Z}}

\newcommand{\GG}{\mathcal{G}}

\newcommand{\la}{\langle}
\newcommand{\ra}{\rangle}

\xyoption{line}

\title{On Central-Peripheral Appendage Numbers of Uniform Central Graphs}
\author{Sul-Young Choi and Jonathan Needleman}
\date{}

\begin{document}
\maketitle
\begin{abstract}
	In a uniform central graph (UCG) the eccentric verticies of a central vertex is the same for all central verticies.  This collection of eccentric verticies is the centered periphery.  For a pair of graphs $(C, P)$ the central-peripheral appendage number, $A_{ucg}(C, P)$, is the minimum number verticies needed to be adjoined to the graphs $C$ and $P$ in order to construct a uniform central graph $H$ with center $C$ and centered-periphery $P$.  We compute $A_{ucg}(C,P)$ in terms of the radius and diameter of $P$ and whether or not $C$ is a complete graph. In the process we show $A_{ucg}(C, P)\leq 6$  if $\text{diam}(P)>2$.   We also provide structure theorems for UCGs in terms of the centered periphery.

\end{abstract}
\section{Introduction}
 Let $G=(V(G), E(G))$ be a simple graph.  The eccentricity of a vertex  $v$, denoted by  $e(v)$, is defined by $max\{d(u,v): u \in V(G)\}$ where $d(u, v)$ is the distance between two vertices $u$ and $v$.   A vertex $x$  is called an eccentric vertex of $v$ if  $d(v,x) = e(v)$; and the set of all eccentric vertices of $v$ is denoted by $EC(v)$. The radius of $G$, $r(G)$ or $r$, is the minimum eccentricity of the vertices in $G$ and the diameter of $G$, $\text{diam}(G)$, is the maximum eccentricity of the vertices in $G$.
A $(u,v)$-path  is called a distance measuring path, or a dmpath, when the length of the path is $d(u,v)$. For a central vertex $u$, a $(u,v)$-dmpath is called a radial path if the length of the path is $r(G)$. For any nonempty subset $S$ of vertices in $G$,  $\la S \ra$ represents the induced subgraph of $G$ by $S$.  For terminology not defined in this paper, the reader is referred to \cite{West}.

Let $\ZZ(G)$ be the center of the graph $G$  and define the {\it centered periphery} of $G$ as   \[\CP(G)=\bigcup_{z\in \ZZ(G)} EC(z),\]

\noindent the set of vertices that are far from the center.    The periphery, which is the set of maximally eccentric vertices, coincides with centered periphery for some graphs.  However, in many cases they differ.  For instance, the graph in figure \ref{fig:P<>CP} has periphery $\{p_0, p_1, p_2, p_5, p_6, p_7\}$ while the  centered periphery is $\{p_1, \ldots, p_6\}$.

\begin{figure}[!htbp]
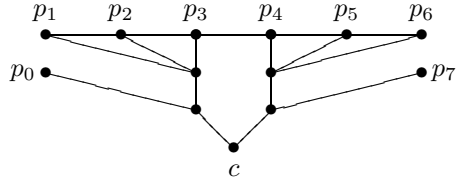

	\[\xy
	(-10,30)*{\bullet};
	%(-20,20)*{\bullet};(-10,20)*{\bullet};(20,20)*{\bullet};(0,20)*{\cdots};
	(-10, 27)*{c};
	%D_1
	(-15,35)*{\bullet};(-5,35)*{\bullet};%(-35,35)*{\bullet};
	%D_2
	(-15,40)*{\bullet};(-5,40)*{\bullet};(-35,40)*{\bullet};(15, 40)*{\bullet};
	(-15,35)*{};(-15,40)*{}**\dir{-};(-5,35)*{};(-5,40)*{}**\dir{-};(-15,35)*{};(-35,40)*{}**\dir{-};(-5,35)*{};(15,40)*{}**\dir{-};
	(-19,40)*{};(0,40)*{};(20,40)*{};(-38, 40)*{p_0};(18, 40)*{p_7};
	%D_3
	(-15,45)*{\bullet};(-5,45)*{\bullet};
	(-15,40)*{};(-15,45)*{}**\dir{-};(-5,40)*{};(-5,45)*{}**\dir{-};
	(-35,48)*{p_1};(-25,48)*{p_2};(-15,48)*{p_3};(-5,48)*{p_4}; (5,48)*{p_5};(15,48)*{p_6};
	%Connecting to D_1
	(-10,30)*{};(-15,35)*{}**\dir{-};(-10,30)*{};(-5,35)*{}**\dir{-};%(-5,30)*{};(-35,35)*{}**\dir{-}; (-38, 40)*{p_0};
	%P
	(-35,45)*{\bullet};(-25,45)*{\bullet};(-15,45)*{\bullet};(-5,45)*{\bullet};(5,45)*{\bullet};(15,45)*{\bullet};
	(-35,45)*{};(-15,40)*{}**\dir{-};(-25,45)*{};(-15,40)*{}**\dir{-};(-15,45)*{};(-15,40)*{}**\dir{-};
	(5,45)*{};(-5,40)*{}**\dir{-};(15,45)*{};(-5,40)*{}**\dir{-};(-5,45)*{};(-5,40)*{}**\dir{-};
	(-35,45)*{};(-25,45)*{}**\dir{-};(-25,45)*{};(-15,45)*{}**\dir{-};(-15,45)*{};(-5,45)*{}**\dir{-};
	(-5,45)*{};(5,45)*{}**\dir{-};(5,45)*{};(15,45)*{}**\dir{-};
	\endxy\]
	\caption{Periphery $\neq$ Centered Periphery}\label{fig:P<>CP}
\end{figure}

 A graph $G$ is called a {\it uniform central graph}, or UCG, if $EC(c)$ is same  for all central vertices $c$ of $G$, i.e. $EC(c)=\CP(G)$ for every vertex $c$ in the center.  It is known that a radial path in a uniform central graph  contains only one central vertex \cite{Choi Manickam}.   In the study of UCGs it is also useful to define the set of  ``intermediate'' vertices of the graph as $\II(G)=V(G)-(\ZZ(G)\cup \CP(G))$.

An appendage number of a graph $G$ is the minimum number of vertices to be added to $G$ to obtain a supergraph $H$ of $G$ so that $H$ satisfies some prescribed  properties.  Buckley, Miller and Slater \cite{Buckley}  studied the appendage number of a graph $G$ satisfying that $G$ is the center of its supergraph.  This result is extended by Gu  \cite{Gu}, where the  supergraph is a UCG with a given $G$ as its center.   More recently S. Klav\v{z}ar, K. Narayankar, and S. Lokesh \cite{Klav} looked at appendage numbers of graphs where the supergraph is a UCG with $G$ as any subgraph.

In this paper, we study the appendage number for a pair of given graphs $(C, P)$, where the supergraph $H$ is a uniform central graph satisfying  $\langle \ZZ(H) \rangle =C$ and $\langle \CP(H) \rangle = P$.  We denote the {\it central-peripheral appendage number}, $A_{ucg}(C,P)$, as the minimum number of   ``intermediate'' vertices needed to construct such a uniform central graph $H$. By convention $A_{ucg}(C, P)=\infty$ if there is no graph satisfying the above conditions.

Gu's results in \cite{Gu} depend on whether or not the desired center is a complete graph.  This distinction also appears in our work.  In light of this we organize our paper in the following way.  Section \ref{sec:GC} focuses on results that do not depend on the center.  These results are then applied to graphs with specific centers; complete graphs in section \ref{sec:C=Kn} and all other graphs in section \ref{sec:C<>Kn} to classify central-peripheral appendage numbers.  The results in sections \ref{sec:C=Kn} and \ref{sec:C<>Kn} are given in terms of sizes of various coverings of $P$.  In section \ref{sec:cov} computes the sizes of these coverings.  Section \ref{sec:Append} summarizes the previous work and describes the central-peripheral appendage numbers in terms of the diameter and radius of $P$.  Finally, in section \ref{sec:Gu}, using different techniques than Gu, we  obtain her results as a corollary of our results.

\section{General Centers}\label{sec:GC}

%\begin{prop}\label{r>1}
%If $P$ is a graph with $r(P)\leq 1$, then $A_{ucg}(C,P)=\infty$ for all graphs $C$.
%\end{prop}
%\begin{proof}
%For given graphs $C$ and $P$ with $r(P)\leq 1$, suppose there exists  a uniform central graph $H$ with $\langle \ZZ(H)\rangle =C$ and $\langle \CP(H) \rangle =P$.  For two vertices $c$ in $ \ZZ(H)$ and $p$ in $\ZZ(P)$, let $\pi$ be an $(c,p)$-radial path, and $x$ the vertex on $\pi$ adjacent to $c$.   Then $d(x,p)=r(H)-1$ and  $x$ is not a central vertex of $H$ since a radial path in a UCG contains only one central vertex.

%For a vertex $y$ in $V(H)- V(P)$, $d(c,y)\leq r(H)-1$ and so
%\[d(x,y) \leq d(x,c)+d(c,y) \leq 1+(r(H)-1)=r(H).\]   For a vertex $q$ in
%$V(P)$, $d(p,q) \leq r(P)\leq 1$ and so
%\[d(x,q) \leq d(x,p) + d(p,q) \leq (r(H)-1)+1=r(H).\]  This implies that $e(x)\leq r(H)$ and $x$ is a central vertex of $H$, a contradiction.

%\end{proof}

%\begin{cor}\label{cor:r>1}
%If $H$ is a UCG with $P=\la\CP(H)\ra$, then $r(P)>1$.
%\end{cor}

%\begin{proof}
%Let $H$ be a UCG with $P=\la\CP(H)\ra$ and $C=\la\ZZ(H)\ra$. By definition $A_{ucg}(C,P)\leq %|\II(H)|<\infty$, and so by proposition \ref{r>1} $r(P)>1$.
%\end{proof}

This section develops results about the structure of a uniform central graph $H$ in terms of $P=\la\CP(H)\ra$.  We do this by studying coverings of the centered periphery.

A \textit{covering} of a graph $G=(V,E)$ is a set $\overline{V}=\{V_1, \ldots, V_k\}$ where $V_i\subset V$ with $\cup V_i= V$.  We say $k$ is the size of the covering.  A \textit{subcovering} of a covering  $\overline{V}$ is a subset of $\overline{V}$ that is also a covering.

Throughout this paper we are  interested in coverings satisfying various properties.  The simplest and most important of these cinditions is condition A.  Let $P$ be a graph with covering $\overline{P}=\{P_1, \ldots, P_k\}$.  We say the pair $(P, \overline{P})$ satisfies
    \begin{description}
        \item[Condition A:] if for each $1\leq i \leq k$, there is  a vertex $p \notin P_i$ satisfying $d(P_i, p)\geq 2$.
    \end{description}
When there is no confusion to the graph $P$, we simply say $\overline{P}$ satisfies condition A.

Elements of a cover can overlap.  It is useful to minimize this overlap in some sense.

\begin{lemma}\label{lem:subcover}
	Let $\{P_1, \ldots, P_k\}$ be a covering of $P$ satisfying condition A.  Then there is a subcovering $\{P_1', \ldots, P_{\eta}'\}$ satisfying condition A and for each $i$ ($1\leq i\leq \eta)$
	\[P_i'\not\subset \bigcup_{j\neq i} P_j'.\]
	That is, for each $i$ there exists $\tilde p_i\in P_i$ such that $\tilde p_i\not\in P_j$ for all $j\neq i$.
\end{lemma}
\begin{proof}
	Suppose there exists an $i$ satisfying
	\[P_i\subset \bigcup_{j\neq i} P_j.\]
	Then $\{P_j: j\neq i\}$ is a subcovering satisfying condition A.
\end{proof}

The main idea of this paper is to take the question of determining $A_{ucg}(C, P)$ and to transfer it to questions about the size of coverings of $P$ satisfying various conditions.  This is possible because for any UCG there exists a natural covering on $P$.   It turns out this covering satisfying condition A.  These coverings are used to obtain lower and upper bounds on $A_{ucg}(C,P)$. To do this we first introduce some notation.

For a uniform central graph $H$ let
\[D_m=\{u\in V(G): d(u, \ZZ(H))=m\}\] \noindent
for $0\leq m \leq r=r(H)$. This gives a stratification of $H$ with $D_0=\ZZ(H)$ and $D_r=\CP(H)$.  For an enumeration $D_1=\{x_1, \ldots, x_k\}$ let $\LL_i$ be the set of all radial paths from $C$ to $P$ containing $x_i$

\begin{lemma}\label{lem:MCondI}
Let $H$ be a uniform central graph with $\langle \CP(H) \rangle =P$.  Then the graph $H$ induces a covering on $P$ satisfying condition A.
\end{lemma}
\begin{proof}
Let $H$ be a uniform central graph with $\langle \CP(H) \rangle =P$. Furthermore, let $C=\langle \ZZ(H) \rangle$ and $D_1=\{x_1, \ldots, x_k\}$ for some $k$.   Let $P_i=V(P)\cap V(\LL_i)$. Then $\{P_1, \ldots, P_{k}\}$ is a covering of $P$ since each radial path from $C$ to $P$ contains a vertex in $D_1$.

Assume $\{P_1,\ldots, P_{k}\}$ does not satisfy the condition A.  Then there is an $i$ satisfying that  for every vertex $p \not\in P_i$ there is a $p'\in P_i$ with  $d(p, p')\leq 1$.  This shows $x_i\in V(C)$, a contradiction.

Since $x_i\in D_1$, there is a central vertex $c$ with $d(x_i, c)=1$.
For each vertex $v$ of $H$ not in $P$, $d(c, v)\leq r(H)-1$ and hence $d(x_i, v)\leq r(H)$.
If $p$ is a vertex in $P_i$, $d(x_i, p)=r(G) -1$ from the construction of $P_i$.

If $p$ is a vertex in $P$ but not in  $P_i$.  Then there is a vertex $p'\in P_i$ with $d(p, p')=1$ and so
\[d(x_i, p)\leq d(x_i, p')+d(p', p)=(r(H)-1)+1=r(H)\]
and so $x_i\in V(C)$, a contradiction.
\end{proof}

\begin{cor}\label{cor:r>1}
	If $H$ is a UCG with $P=\la\CP(H)\ra$, then $r(P)>1$.
\end{cor}

\begin{proof}
	If $r(P)=1$, then there is $q\in V(P)$ with $e(q)=1$.  Let $\overline{P}=\{P_1, \ldots, P_k\}$ be the induced cover of $P$.  Without loss of generality we may assume $q\in P_1$.  Then for all $p\in V(P)$, $d(p, P_1)\leq 1$ and condition A is not met.  Hence by lemma \ref{lem:MCondI} $P$ is not a centered periphery for any UCG.
\end{proof}

\begin{cor}\label{r>1}
	If $P$ is a graph with $r(P)\leq 1$, then $A_{ucg}(C,P)=\infty$ for all graphs $C$.
\end{cor}

For $H$, a UCG with $C=\la\ZZ(H)\ra$  and $P=\la\CP(H)\ra$, we use the induced covering from lemma \ref{lem:MCondI} to gain a better understanding of the structure of $\II(H)$.

 Let  $D_1=\{x_1, \ldots, x_k\}$ and $\{P_1, \ldots, P_{k}\}$ be the induced covering of $P$ defined as in lemma \ref{lem:MCondI}. Without loss of generality assume $\{P_1,\ldots, P_{k'}\}$ with $k'\leq k$, is subcovering with non-empty elements .  Also let  $\overline{P}=\{P_1, \ldots, P_{k''}\}$ and $\{\tilde p_1, \ldots, \tilde p_{k''}\}$, with $k''\leq k'$,  be as in lemma \ref{lem:subcover}.  For $1\leq m\leq r(H)-1$ we define the following subsets of $D_m$.

 \[D_m'=D_m\cap \bigcup_{i=1}^{k'}V(\LL_i) \text{ and } D_m''(\overline{P})=D_m\cap \bigcup_{i=1}^{k''}V(\LL_i). \]
$D_m$ and $D_m'$ are well-defined with respect to $H$, however $ D_m''(\overline{P})$ depends on a choice of subcover.  When this choice is  clear from context, we simply use $D''_m$.

 The following proposition describes the structure of $\II(H)$ that is the key to the rest of this paper.

 \begin{prop}\label{prop:uniquepaths}
 	Let $H$ be a UCG with $C=\la\ZZ(H)\ra$,  $P=\la\CP(H)\ra$,   $\overline{P}=\{P_1, \ldots, P_{k''}\}$  and $\{\tilde p_1, \ldots, \tilde p_{k''}\}$ as in lemma \ref{lem:subcover} with respect to the induced covering from lemma \ref{lem:MCondI}.  Then any $(x_i, \tilde p_i)$-dmpath and $(x_j, \tilde p_j)$-dmpath are disjoint if $i\neq j$.	
 \end{prop}
\begin{proof}
	For  $L_i$ a $(x_i, \tilde p_i)$-dmpath and $L_j$ a $(x_j, \tilde p_j)$-dmpath with $i\neq j$, assume there exists a  $y\in V(L_i)\cap V(L_j)$.  Let $L_i'$ be the subpath of $L_i$ from $x_i$ to $y$ and $L_j'$ the subpath of $L_j$ from $y$ to  $\tilde p_j$.  The concatenation   of $L_i'$ and $L_j'$ yields a $(x_i, \tilde p_j)$-dmpath and so $\tilde p_j\in P_i$, a contradiction.
\end{proof}

\begin{cor}\label{cor:Dm}
	With the assumptions of proposition \ref{prop:uniquepaths}, $|D_1''|\leq |D_m''|$ for $1\leq m\leq r(H)-1$.
\end{cor}
\begin{proof}
	For each $1\leq i\leq k''$ there exists $L_i$, a $(x_i, \tilde p_i)$-dmpath.  Let $\LL=\{L_1, \ldots, L_{k''}\}$ and $F_m=V(\LL)\cap D_m''$.  By proposition   \ref{prop:uniquepaths} $|F_1|=|F_m|$.  Then by the construction of $F_m$
	\[|D_1''|=|F_1|=|F_m|\leq |D_m''|.\]
\end{proof}

We often work with UCGs with few intermediate verticies. For such UCGs we can  say more about thier structure.

\begin{lemma}\label{lem:y_i-P_i}
	Let $H$ be a UCG with $P=\la\CP(H)\ra$ and $r=r(H)$.  Furthermore, let  $\overline{P}=\{P_1, \ldots, P_{k''}\}$ be a subcovering of the induced covering and $\{\tilde p_1, \ldots, \tilde p_{k''}\}$ an associated set of vertices as in lemma \ref{lem:subcover}.  If $|D_1''|= |D_{r-1}''|$ then for each $P_i\in \overline{P}$ there is a unique $y_i\in D_{r-1}''$ such that $y_i$ is adjacent to every vertex in $P_i$.  Furthermore, if $i\neq j$, then $y_i\neq y_j$.
\end{lemma}
\begin{proof}
	Let $D_1''=\{x_1, \ldots, x_{k''}\}$. For each $i$, let $L_i$ be an $(x_i, \tilde p_i)$-dmpath and let $y_i$ be the vertex on $L_i$ adjacent to $\tilde p_i$.  Then $y_i\in D_{r-1}''$.  By proposition \ref{prop:uniquepaths}  $y_i\neq y_j$ when $i\neq j$. Furthermore $D_{r-1}''=\{y_1, \ldots, y_{k''}\}$.  We claim $y_i$ is adjacent to every vertex in $P_i$.
	
	For a vertex $p\in P_i$ there is a $c\in \ZZ(H)$ and a $(c, p)$-radial path $L$ that contains $x_i$. Let $y$ be the vertex on $L$ adjacent to $p$.   Since $y\in D_{r-1}''$, $y=y_j$ for some $1\leq j\leq k''$.  Since $y_j$ is adjacent to $\tilde p_j$, there is a $(x_i, \tilde p_j)$-dmpath of length $r-1$. Therefore, $\tilde p_j\in P_i$ which means $j=i$, $y=y_i$ and $y_i$ is adjacent to $p$.
\end{proof}

\begin{lemma}\label{lem:C-D_1}
	Let $H$ be a UCG with $C=\la\ZZ(H)\ra$ and $P=\la\CP(H)\ra$ and assume the induced covering $\overline{P}=\{P_1, \ldots, P_k\}$ satisfies the conditions of the subcovering in  lemma \ref{lem:subcover}.  That is $D_1=D_1''$.  Then  each vertex in $C$ is adjacent to each vertex in $D_1$.
\end{lemma}
\begin{proof}
	Let $D_1=\{x_1, \ldots, x_k\}$ and  $\{\tilde p_1, \ldots, \tilde p_k\}$ be an associated set of vertices to $\overline{P}$ as in lemma \ref{lem:subcover}.  For each vertex $c$ in the center and each $\tilde p_i$, there is a $(c, \tilde p_i)$-radial path.  By construction of $\overline{P}$ and definition of $\tilde p_i$, each $(c, \tilde p_i)$-radial paths must contain $x_i$.  Therefore, $c$ and $x_i$ are adjacent.
\end{proof}

For a graph $P$, let $\text{cov}_A(P)$ be the smallest size of a covering of $P$ satisfying condition A.    Observe, from the definition of condition A it follows that  $\text{cov}_A(P)\geq 2$ for any graph $P$ with $r(P)>1$.  The following result is crirtical in obtaining lower bounds on $A_{ucg}(C, P)$.

\begin{prop}\label{Lboundk}
 If $H$ is a uniform central graph with $C=\langle \ZZ(H) \rangle$, $\langle \CP(H) \rangle=P$, $r=r(H)$ and $\kappa=\text{cov}_A(P)$, then  $|\II(H)|\geq \kappa(r-1)$.
\end{prop}

\begin{proof}
Let $\{P_1, \ldots, P_{k''}\}$ be a subcover of the induced cover on $P$ from $H$ from lemmas \ref{lem:MCondI} and  \ref{lem:subcover}.  By definition of $\kappa$ and corollary \ref{cor:Dm} it follows that
\[\kappa\leq k''=|D_1''|\leq |D_m''|\leq |D_m|.\]
Since \[\II(H)=\bigcup_{m=1}^{r-1} D_m\]
the result follows.
\end{proof}

Proposition \ref{Lboundk} is used to obtain bounds on the radius of a UCG.

\begin{cor}\label{cor:rbound}
	If $H$ is a uniform central graph with $\langle \CP(H) \rangle=P$, $\kappa=\text{cov}_A(P)$ and $|\II(H)|\leq s\kappa+t$ where $s, t\in \Nb$ and $t<\kappa$, then $r(H)\leq s+1$.
\end{cor}
\begin{proof}
	Observing $(s+1)\kappa >s\kappa+t$ the result follows from the contrapositive of  proposition \ref{Lboundk}.
\end{proof}

We now  construct a graph with center $C$ and centered periphery $P$. Let $C$ and $P$ be graphs, and $\rho\in \Nb$. Suppose $\overline{P}=\{P_1, \ldots, P_k\}$ is a covering of $P$.  Define a graph $G=\GG(C, P, \overline{P}, \rho)$ as follows:   \[V(G)=V(C) \cup V(P) \cup \{x_{i,j}: \, 0\leq i\leq k, 1\leq j\leq \rho\}\]
and  $ab$ is an edge in $G$ if and only if one of the following occurs
\begin{enumerate}
    \item $ab$ is an edge of $C$.
    \item $ab$ is an edge of $P$.
    \item $a  $ is a vertex of $C$ and $b=x_{i, 1}$ for some $i$ with $0\leq i \leq k$.
    \item $a\in P_i$ and $b=x_{i,\rho}$ for some $i$ with $1\leq i\leq k$.
    \item $a=x_{i,j}$ and $b=x_{i,j+1}$ for $i$ and $j$ with $0\leq i\leq k, 1\leq j\leq \rho-1$
\end{enumerate}

\begin{figure}[!htbp]
\[\xy
(0,10)*\ellipse(20,4){-};
%(-20,20)*{\bullet};(-10,20)*{\bullet};(20,20)*{\bullet};(0,20)*{\cdots};
(0,20)*{C};
%D_1
(-15,35)*{\bullet};(-5,35)*{\bullet};(8,35)*{\cdots};(15,35)*{\bullet};(-35,35)*{\bullet};
(-19,35)*{x_{1,1}};(0,35)*{x_{2,1}};(20,35)*{x_{k,1}};(-39,35)*{x_{0,1}};
%D_2
(-15,40)*{\bullet};(-5,40)*{\bullet};(8,40)*{\cdots};(15,40)*{\bullet};(-35,40)*{\bullet};
(-15,35)*{};(-15,40)*{}**\dir{-};(-5,35)*{};(-5,40)*{}**\dir{-};(15,35)*{};(15,40)*{}**\dir{-};(-35,35)*{};(-35,40)*{}**\dir{-};
(-19,40)*{x_{1,2}};(0,40)*{x_{2,2}};(20,40)*{x_{k,2}};(-39,40)*{x_{0,2}};
%D_3
(-15,45)*{\bullet};(-5,45)*{\bullet};(8,45)*{\cdots};(15,45)*{\bullet};(-35,45)*{\bullet};
(-15,40)*{};(-15,45)*{}**\dir{.};(-5,40)*{};(-5,45)*{}**\dir{.};(15,40)*{};(15,45)*{}**\dir{.};(-35,40)*{};(-35,45)*{}**\dir{.};
(-19,45)*{x_{1,\rho}};(0,45)*{x_{2,\rho}};(20,45)*{x_{k,\rho}};(-39,45)*{x_{0,\rho}};
%Connecting to D_1
%(-20,20)*{};(-15,35)*{}**\dir{-};(-20,20)*{};(-5,35)*{}**\dir{-};(-20,20)*{};(15,35)*{}**\dir{-};(-20,20)*{};(-35,35)*{}**\dir{-};
(-12,25)*{};(-15,35)*{}**\dir{-};(-6,25)*{};(-15,35)*{}**\dir{-};(0,25)*{};(-15,35)*{}**\dir{-};
(-3,25)*{};(-5,35)*{}**\dir{-};(3,25)*{};(-5,35)*{}**\dir{-};(9,25)*{};(-5,35)*{}**\dir{-};
(6,25)*{};(15,35)*{}**\dir{-};(10,25)*{};(15,35)*{}**\dir{-};(14,25)*{};(15,35)*{}**\dir{-};
(-20,25)*{};(-35,35)*{}**\dir{-};(-14,25)*{};(-35,35)*{}**\dir{-};(-8,25)*{};(-35,35)*{}**\dir{-};
%P
%(0,28)*\ellipse(30,6){-};
(0, 66)*{P};
(-26,57)*{P_1};
(-15,45)*{};(-28,55)*{}**\dir{-};(-25,55)*{};(-15,45)*{}**\dir{-};(-15,45)*{};(-22,55)*{}**\dir{-};
(-15,57)*{P_2};
(-5,45)*{};(-18,55)*{}**\dir{-};(-15,55)*{};(-5,45)*{}**\dir{-};(-5,45)*{};(-12,55)*{}**\dir{-};
(24,57)*{P_k};
(15,45)*{};(28,55)*{}**\dir{-};(25,55)*{};(15,45)*{}**\dir{-};(15,45)*{};(22,55)*{}**\dir{-};
(5,55)*{\cdots};
(-13,28.3)*\ellipse(6.5,3){-};(12,28.3)*\ellipse(6.5,3){-};(-7.5,28.3)*\ellipse(6.5,3){-};
(-34,62)*{};(27,62)*{};**\frm{^)};
\endxy\]
\caption{$G=\GG(C, P, \overline{P},\rho)$}\label{fig:G0}
\end{figure}
\begin{prop}\label{G0UCG}
 For two given graphs  $C$ and $P$, let $\overline{P}=\{P_1, \ldots, P_k\}$ be a covering of $P$ satisfying condition A, and   $\rho \geq \text{min}\{\text{diam}(C), 2\}$.  Then $G=\GG(C,P, \overline{P}, \rho)$ is a  UCG with radius $\rho+1$, $\langle \ZZ(G) \rangle = C$ and $\langle \CP(G) \rangle  =P$.
\end{prop}

\begin{proof}
From the construction of $G$, we can show the following:
\begin{enumerate}[i)]
    \item For all $j=1, \dots, \rho$ and any vertex $p$ in $P$, $d(x_{0,j},p)\geq \rho+2$, and so $e(x_{0,j}) \geq \rho+2$ and $e(p) \geq \rho+2$.
    \item For all $i$ and $j$ with $1\leq i \leq k$ and $2 \leq j \leq \rho$, $d(x_{i,j}, x_{0, \rho})\geq \rho+2$, and so $e(x_{i,j}) \geq \rho+2$.
    \item Since the covering $\overline{P}$ satisfies condition A, for each $i$ with $1\leq i \leq k$ there exists a vertex $p$ in $P_j$ for some $j \neq i$ satisfying $d_P(P_i, p) \geq 2$.  Thus $d(x_{i,1}, p) \geq \rho+2$ and  $e(x_{i,1}) \geq \rho+2$.
    \item For a vertex $c$ in $C$,  $d(c,x_{i,j})\leq \rho$, $d(c,p)=\rho+1$ for all $p$ in $P$, and  $d(c,c')\leq 2$ for any vertex $c'$ in $C$.  This implies $e(c)=\rho+1$ and $EC(c)=V(P)$.
\end{enumerate}
\noindent
Therefore, $G=\GG(C, P, \overline{P}, \rho)$ is a  UCG with radius $\rho+1$, $\la\ZZ(G)\ra=C$ and $\langle \CP(G) \rangle =P$.

\end{proof}

We end this section with a result about conditions when a spanning subgraph of a UCG is still a UCG.

\begin{lemma}\label{lem:spanUCG}
	Let $H$ be a UCG with $V(C)=\ZZ(H)$ and let $G$ be a spanning subgraph of $H$.  Then $G$ is a UCG with $V(C)=\ZZ(G)$ and $\CP(G)=\CP(H)$ if for all $c\in V(C)$ and $x\in V(G)$, $d_G(c, x)=d_H(c,x)$.
\end{lemma}
\begin{proof}
	Since $E(G)\subset E(H)$, the  eccentricity of a vertex $v$ in $G$ are greater than or equal to the eccentricity of $v$ in $H$.  The assumption that $d_G(c, x)=d_H(c,x)$ for all $c\in V(C)$ and $x\in V(G)$ implies $\ZZ(G)=V(C)$ and $EC(c)$ in $G$ is $\CP(H)$.  Hence, $G$ is a UCG with $V(C)=\ZZ(G)$ and  $\CP(G)=\CP(H)$.
\end{proof}

\section{When $C=K_n$ with $n\geq 2$}\label{sec:C=Kn}
In this section we compute appendage numbers when the center $C$ is a complete graph $K_n$ with $n\geq 2$  in terms of the size of a smallest covering for the centered periphery.

 We introduce a new condition on coverings.  Let $P$ be a graph with  covering $\overline{P}=\{P_1, \ldots, P_k\}$.  We say $(P, \overline{P})$ satisfies
\begin{description}
	\item[Condition B:] if for every $1\leq i\leq k$ and for each $p\in P_i$ either
	\begin{enumerate}
		\item there is a vertex $ p' \not\in P_i$ with $d(p, p')\geq 3$\label{d3}, or
		\item there is a $j\neq i$ satisfying $d(p, P_j)\geq 2$. \label{Pj=2}
	\end{enumerate}
\end{description}
Once again, we say the covering $\overline{P}$ satisfies condition B when there is no confusion about $P$.  Also define $\text{cov}_{AB}(P)$ to be the smallest size of a covering of $P$ satisfying both condition A and condition B.

Before determining appendage numbers it is necessary to prove the existence of coverings satisfying conditions A and B.

\begin{lemma}\label{lem:covexist}
Let $P$ be a graph with $r(P)\geq 2$. Then there exists a covering of $P$ satisfying conditions A and B.
\end{lemma}
\begin{proof}
Let $V(P)=\{p_1, \ldots, p_k\}$.  Then $\overline{P}=\{\{p_1\}, \ldots, \{p_k\}\}$ is a covering of $P$ and  $\overline{P}$ satisfies conditions $A$ and $B$  because $r(P)\geq 2$.
\end{proof}

We now find the appendage numbers in terms of $\kappa=\text{cov}_A(P)$.

\begin{prop}\label{prop:Knbound}
$\kappa\leq A_{ucg}(K_n, P)\leq \kappa+1$ where $\kappa=\text{cov}_A(P)$.
\end{prop}
\begin{proof}
 By lemma \ref{lem:covexist} there is a covering  of $P$ satisfying condition A, so let $\overline{P}=\{P_1, P_2, \ldots, P_\kappa\}$ be  a smallest covering of $P$ satisfying condition A. By proposition \ref{G0UCG}, the graph $G=\GG(K_n, P, \overline{P}, 1)$ is a UCG with radius 2 and $|\II(G)|=\kappa+1$, and thus $A_{ucg}(K_n, P)\leq \kappa+1$.  The lower bound is from proposition \ref{Lboundk}.

\end{proof}

\begin{thm}\label{thm:Knrange}
Let $\kappa=\text{cov}_A(P)$.  Then $A_{ucg}(K_n, P)=\kappa$ if and only if $\text{cov}_{AB}(P)=\kappa$.
\end{thm}

\begin{proof}
Suppose $\text{cov}_{AB}(P)=\kappa$ and let $\overline{P}=\{P_1, \ldots, P_\kappa\}$ be a smallestl covering of $P$ with respect to conditions A and B.  Also, let $G=\GG(K_n, P, \overline{P}, 1)-\{x_{0,1}\}$.  Since $\overline{P}$ satisfies condition B, one can verify that for $G$, $e(p) \geq 3$ for all $p\in V(P)$. Using a similar argument to  proposition \ref{G0UCG} we can show $G$ is a UCG with $r(G)=2$, $\la\CP(G)\ra=P$, $\la\ZZ(G)\ra=K_n$ and  $|\II(G)|=\kappa$.

Now assume $A_{ucg}(K_n, P)=\kappa$.  Let $H$ be a UCG with $P=\langle \CP(H) \rangle , \langle \ZZ(H) \rangle =K_n$ and $|\II(H)|=\kappa$.  From corollary \ref{cor:rbound}, $r(H)=2$.
Let $\{P_1, \ldots , P_{k}\}$ be the induced covering of $P$ as in lemma \ref{lem:MCondI}.  We know
\[\kappa\leq k= |D_1|= |\II(H)|= \kappa\]
 and so $k=\kappa$.  Let  $D_1=\{x_1,\ldots , x_k\}$, where $x_i$ is associated to $P_i$.

 Suppose $p$ is a vertex in $P_i$.   Since $e(p)\geq 3$ there is an $x\in V(H)$ with $d_H(p,x)\geq 3$.   Note, $x\not\in P_i$, because  all vertices in $P_i$ are adjacent to $x_i$.    If $x \in P_j$ for some $j \neq i$, then  condition B-1 is satisfied by definition. Finally, if $x=x_j\in D_1$ for some $j\neq i$, then $d_P(p, P_j)\geq 2$ and so $\overline{P}$ satisfies condition B-2. % Assume $k=\text{cov}_I(M)\neq\text{cov}_{II}(M)$, and we will show $A_{ucg}(K_n, M)\neq k$.
\end{proof}

In section \ref{sec:cov} we determine when $\text{cov}_A(P)=\text{cov}_{AB}(P)$.

\section{When  $C\neq K_n$.}\label{sec:C<>Kn}
For this section assume that $C$ is not a complete graph and so $\text{diam}(C) \geq 2$.  The results in this section mirror those of when $C=K_n$, however conditions on the coverings of $P$, as well as the proofs, are more technical.

\begin{prop}\label{prop:Cbound}
For a given pair of graphs $(C, P)$ , $2\kappa\leq A_{ucg}(C, P)\leq 2\kappa+2$ where $\kappa=\text{cov}_A(P)$.
\end{prop}
\begin{proof}
For a given pair of graphs $C$ and $P$, suppose $H$ is a UCG with $\la\ZZ(H)\ra=C$ and $\la\CP(H)\ra=P$.  Then $r(H)\geq \text{diam}(C)+1\geq 3$.  From corollary \ref{Lboundk} we obtain a lower bound $A_{ucg}(C, P)\geq 2\kappa$.

For the upper bound consider $G=\GG(C, P, \overline{P}, 2)$ where $\overline{P}$ is a smallest covering of $P$ satisfying condition A.  By proposition \ref{G0UCG} $G$ is a UCG with $|\II(G)|=2\kappa+2$ and so  $A_{ucg}(C, P)\leq 2\kappa+2$.
\end{proof}

In the proof of proposition \ref{prop:Cbound} we use the fact that $\GG(C, P, \overline{P}, 2)$ is a UCG so long as $\overline{P}$ satisfies condition A.  We now consider a modification of this graph, and determine conditions that this new graph is a UCG.

Let $\overline{P}=\{P_1, \ldots, P_k\}$ be a covering of $P$  and let
\[G=\GG(C, P, \overline{P}, 2)-\{x_{0,1}, x_{0,2}\}.\]
\noindent  By construction, $d(c, p)=3$ for all $c\in V(C)$,  $p\in V(P)$, and $d(c, x)\leq 2$ for $x\in V(G)-V(P)$.  So $G$ is a UCG with center $C$ and centred periphery $P$  if and only if for all $u\in V(G)-V(C)$, $e(u)\geq 4$.

Consider the case $u=x_{i,1}$.  If $G$ is a UCG, $\overline{P}$ must satisfy condition A.  This implies $e(x_{i, 1})\geq 4$ for all $1\leq i \leq k$.

Next, consider the case $u=x_{i, 2}$ for some $1\leq i \leq k$. Since $e(u)\geq 4 $ there exists a $v\in V(G)$ such that $d(u, v)\geq 4$.  Because $v\not\in V(C)$ and $d(u, x_{j, 1})\leq 3$, either $v\in V(P)$ or  $v=x_{j, 2}$ for some $j\neq i$.  If $p\in V(P)$ then $v\not\in P_i$ and $d(v, P_i)\geq 3$.  If $v=x_{j, 2}$, then $d(P_i, P_j)\geq 2$.

Finally, consider the case $u\in P_i$ for some $i$.  If $G$ is a UCG, then $e(u)\geq 4$, and so there is a $v\in V(G)$ such that $d(u, v)\geq 4$.  Since $v\not\in V(C)$ and $v\not\in P_i$ because $d(p, p')\leq 2$ for all $p, p'\in P_i$, it follows that either $v=x_{j,1}$, $v=x_{j,2}$ or $v\in P_j$ for some $j\neq i$.  If $v=x_{j, 1}$ then $d(u, P_j)\geq 2$ for $j\neq i$.  If $v=x_{j, 2}$ then $d(u, P_j)\geq 3$.  Finally, if $v\in P_j$ for  some $j\neq i$, then for all $p\in P_j$
\[4\leq d(u, v)\leq d(u, p)+d(p, v)\leq d(u, p)+ 2.\]
We conclude $d(u, P_j)\geq 2$. This analysis gives rise to two new conditions on $\overline{P}$.

Let $\overline{P}=\{P_1, \ldots, P_k\}$ be a covering of a graph $P$. We say $(P, \overline{P})$ satisfies if for all $1\leq i\leq k$
\begin{description}
        \item[Condition A$'$:] if either
            \begin{enumerate}
                \item there is  a $p\not\in P_i$ satisfying $d(P_i, p)\geq 3$, or \label{Ib:1}
                \item there is a $j\neq i$ satisfying $d(P_i, P_j)\geq 2$. \label{Ib:2}
            \end{enumerate}
        \item[Condition B$'$:] if for each $p\in P_i$ there is a $j\neq i$ satisfying $d(p, P_j)\geq 2$. \label{IIb:2}

\end{description}
 Condition A$'$ arises from $e(x_{i,2})\geq 4$ and condition B$'$ from $e(p)\geq 4$, for $p\in V(P)$.  Note condition A$'$ implies condition A, and condition B$'$ implies condition B.  We  often abuse notation and say the covering $\overline{P}$ satisfies a specified condition.

Let $\text{cov}_{A'}(P)$ be the smallest size of the  covering of $P$ satisfying condition A$'$ and $\text{cov}_{A'B'}(P)$  the smallest size of the covering of $P$ satisfying conditions A$'$ and B$'$.

The arguments used  to determine conditions A$'$ and B$'$ are bi-directional.  This implies the following proposition.

\begin{prop}\label{prop:A'B'UCG}
Let $C$ and $P$ be graphs, $\overline{P}$  a covering of $P$, and $G=\GG(C, P, \overline{P}, 2)-\{x_{0,1}, x_{0,2}\}$.  Then $G$ is a UCG with center $C$ and centered periphery $P$ if and only if $\overline{P}$ satisfies conditions A$'$ and B$'$.
\end{prop}
We are now ready to relate appendage numbers to coverings.
\begin{prop}\label{prop:Aucg=2k}
	Let $\kappa=\text{cov}_A(P)$. For a graph $P$ with $r(P)>1$,  $\text{cov}_{A'B'}(P)=\kappa$ if and only if $A_{ucg}(C, P)=2\kappa$.
\end{prop}
\begin{proof}
	First, assume $\text{cov}_{A'B'}(P)=\kappa$.  Let $\overline{P}=\{P_1, \ldots, P_\kappa\}$ be a covering of $P$ satisfying conditions A$'$ and B$'$.  Then by proposition \ref{prop:A'B'UCG} the graph \[G=\GG(C, P, \overline{P}, 2)-\{x_{0,1}, x_{0,2}\}\]
	\noindent is a UCG and $|\II(G)|=2\kappa$ and so $A_{ucg}(C, P)\leq 2\kappa$.  However, by proposition \ref{prop:Cbound}, $A_{ucg}(C, P)\geq 2\kappa$ and the result follows.
	
	Next, assume $A_{ucg}(C, P)=2\kappa$.  Let $H$ be a UCG with $P=\la\CP(H)\ra$, $C=\la\ZZ(H)\ra$ and $|\II(H)|=2\kappa$. Note that  $r(H)= 3$ by corollary \ref{cor:rbound}.
	
	Let $\{P_1, \ldots , P_{k}\}$ be the induced covering of $P$ as in lemma \ref{lem:MCondI}.   There exists a subcover $\overline{P}=\{P_1, \ldots, P_{k''}\}$ and an associated set of vertices $\{\tilde p_1, \ldots, \tilde p_{k''}\}$ as in lemma \ref{lem:subcover}.   Assume $D_1''=\{x_1, \ldots, x_{k''}\}$, and by corollary \ref{cor:Dm}  $|D_2''|\geq k''$.
	%\{y|y \text{ is on a dmp from} C \text{ to } p_i\text{ for } 1\leq i\leq k''\}\]
	Because $\kappa=\text{cov}_A(P)$ and $\overline{P}$ satisfies condition A we know
	\[2\kappa\leq 2k''\leq  |D_1''|+|D_2''|\leq |\II(H)|= 2\kappa,\]
	and so $k''=|D_1''|=|D_2''|=\kappa$.  This implies $D_1=D_1'', D_2=D_2''$ and $\kappa=k''=k$.
	
	We now show $H$ contains a spanning subgraph isomorphic to  \[G=\GG(C, P, \overline{P}, 2)-\{x_{0,1}, x_{0,2}\}.\]
	\noindent  First, for each $\tilde p_i$ defined above and each central vertex $c$ there is a $(c, \tilde p_i)$-radial path.  By construction of $\tilde p_i$ this path must contain $x_i$, and thus $x_i$ is adjacent to every $c$ in the center by lemma \ref{lem:C-D_1}.
	
	By lemma \ref{prop:uniquepaths} there is an enumeration of $D_2=\{y_1, \ldots, y_k\}$ so that  $y_i$ is adjacent to both $x_i$ and each vertex in $P_i$.   Therefore, $G=\GG(C, P, \overline{P}, 2)-\{x_{0,1}, x_{0,2}\}$ is isomorphic to a spanning subgraph of $H$.  By lemma \ref{lem:spanUCG}, $G$ is a UCG and by proposition \ref{prop:A'B'UCG}, $\overline{P}$ satisfies conditions A$'$ and B$'$.
	
\end{proof}

We now move on to understand when $A_{ucg}(C, P)= 2\kappa+1$.  Let $H$ be a UCG such that $\la\ZZ(H)\ra= C$, $\la\CP(H)\ra=P$ and $|\II(H)|=2\kappa+1$.  By proposition \ref{prop:Aucg=2k}, $\kappa\neq \text{cov}_{A'B'}(P)$.  Also, by corollary \ref{cor:rbound}  $r(H)=3$, and hence $|D_1|+|D_2|=2\kappa+1$.  Let $\{P_1, \ldots, P_{k}\}$ be the induced covering of $P$ through $D_1=\{x_1, \ldots, x_k\}$ as in lemma \ref{lem:MCondI}. Without loss of generality let $\overline{P}=\{P_1, \ldots, P_{k''}\}$, with $k''\leq k$, be a subcover with an associated set of vertices $\{\tilde p_1, \ldots, \tilde p_{k''}\}$ as in lemma \ref{lem:subcover}. Then

\[2\kappa\leq 2k''\leq  |D_1''|+|D_2''|\leq |\II(H)|=2\kappa+1\]
and hence $\kappa=k''$.

Since \[\kappa=|D_1''|\leq |D_1| \text{ and } \kappa=|D_1''|\leq |D_2''|\leq |D_2|,\]  either $|D_1|=\kappa+1$ or $|D_2|=\kappa+1$.

We first address when $|D_1|=\kappa+1$ and $|D_2|=\kappa$.
\begin{prop}\label{prop:2k+1_case1}
	Assume $A_{ucg}(C, P)= 2\kappa+1$, and let $H$ be a UCG such that $\la\ZZ(H)\ra= C$, $\la\CP(H)\ra=P$, and $|\II(H)|=2\kappa+1$ where $\kappa=\text{cov}_A(P)$.  Furthermore, assume $|D_1|=\kappa+1$.  Then $\text{cov}_{A'}(P)=\kappa$.
\end{prop}

\begin{proof}
By corollary \ref{cor:Dm} it follows that
\[\kappa=|D_1''|\leq |D_2''|\leq |D_2|= \kappa.\]
Therefore $D_2=D_2''$ and  $|D_1''|=\kappa$. Since $k=|D_1|=\kappa+1$, $x_{k}\in D_1$ but $x_{k}\not\in D_1''$.

Next, we prove $x_k\not\in D_1'$ by showing if $x_k\in D_1'$, then $\overline{P}$ satisfies both  conditions A$'$ and B$'$. Hence $A_{ucg}(C,P)=2\kappa$ by proposition \ref{prop:Aucg=2k}.

Assume $x_k\in D_1'$, that is $D_1=D_1'$.  Because $|D_2|=|D_2''|=\kappa=|D_1''|$ we may assume $D_2=\{y_1, \ldots, y_\kappa\}$ such that each $y_i$ is adjacent to $x_i$ and all vertices of $P_i$ by lemma \ref{lem:y_i-P_i}.  To understand the structure of $P_k$  define the indexing set

\[I=\{i: x_k \text{ is adjacent to } y_i \text{ for some } y_i\in D_2\}.\]

From the definition of $D_1'$ and  $I$,
\[P_k=\bigcup_{i\in I}P_i.\]
For each $j\not \in I$ every radial path to $\tilde p_j$ must contain $x_j$.  Hence, each central vertex $c\in V(C)$ is adjacent to $x_j\in D_1$ for $j\not\in  I$.  A similar argument shows that for each $c\in V(C)$ and each $i\in I$, $c$ is adjacent to either $x_k$ or $x_i$.

We now show $\overline{P}$ satisfies conditions A$'$ and B$'$.

First, assume condition B$'$ fails.  Then there exists an $\iota $ and a vertex $p\in P_\iota$ such that $d_P(p, P_j)\leq 1$ for all $j\neq \iota$. Therefore,  there is a $p_j\in P_j$ such that $d(p, p_j)\leq 1$, and so the following hold:

\begin{enumerate}[i)]
	\item $d(p, x_{j})\leq  d(p, p_j)+d(p_j, x_{j})\leq 3$.
	\item $d(p, y_{j})\leq  d(p, p_j)+d(p_j, y_{j})\leq 2$.
	\item $d(p,p')\leq  d(p, p_j)+d(p_j, p')\leq 3$ for all $p'\in P_j$.
	\item $d(p, x_{\iota})=2$.
	\item $d(p, y_{\iota})=1$.
	\item $d(c, p)=3$  for $c\in V(C)$.
\end{enumerate}
If $\iota \in I$, then $d(p, x_k)=2$. If $\iota\not\in I$,  then for a $j\in I$ \[d(p, x_k)\leq  d(p, p_j)+d(p_j, x_k)\leq 3\]
\noindent since $d(p_j, x_k)=2$. Therefore $e(p)=3$, a contradiction, and so $\overline{P}$ satisifies condition B$'$.

Next, assume condition A$'$ fails.  Then there exists an $\iota$ such that $d_P(P_\iota, p')\leq 2$ for all $p'\in V(P)-P_i$ and $d_P(P_\iota, P_j)\leq 1$ for all $j\neq \iota$.  Then $d(y_{\iota}, p)\leq 3$ for all $p\in V(P)$  and $d(y_{\iota}, y_{j})\leq 3$ for all $j\neq \iota$.  We obtain a contradiction by showing $y_\iota$ is in the center.

If $\iota\not\in I$ then for each $x_j\in D_1$ there is a $c\in V(C)$ that is adjacent to $x_j$. Then $y_\iota-x_\iota-c-x_j$ is a path and $d(y_\iota, x_j)\leq 3$.  Similarly when $\iota\in I$ and  $j\not\in I$, $d(y_\iota, x_j)\leq 3$.  If $\iota, j\in I$ then $y_\iota-x_k-y_j-x_j$ is a path and so $d(y_\iota, x_j)\leq 3$.  Finally,  $d(y_\iota, c)=2$ for all $c\in V(C)$, and so $e(y_{\iota})=3$, a contradiction and $\overline{P}$ satisfies condition A$'$.

Since $\overline{P}$ satisfies conditions A$'$ and B$'$,  $A_{ucg}(C, P)=2\kappa$ by proposition \ref{prop:Aucg=2k}.   This contradicts the assumptions of the proposition and hence $x_k\not\in D_1'$ and so $D_1'=D_1''$.

Then $x_k$ is not on a radial path because $x_k\not\in D_1'$.  Therefore, vertices adjacent to  $x_k$ are in $C$ or $D_1''$.  Furthermore each $y_i\in D_2$ satisfies $e(y_i)\geq 4$ from  $H$ being  a UCG with $r(H)=3$  .  Therefore, there exists a $u\in V(G)$ with $d(y_i, u)\geq 4$. Because $d(c, y_i)= 2$ for all $c\in V(C)$, we know  $u\not\in V(C)$.

Because $D_1'=D_1''$, each $(c, \tilde p_i)$-radial path contains $x_i$, and hence each $c$ is adjacent to $x_i$.  It follows that for a $c\in V(C)$ adjacent to $x_j\in D_1$ that \[d(y_i, x_j)\leq d(y_i, c)+d(c, x_j)=3,\] \noindent  which means $u\not\in D_1$.

If $u \in V(P)$ then $d(u, P_i)\geq 3$ and condition A$'$-\ref{Ib:1} is satisfied for $i$.  If $u=y_j\in D_2''$ then $d_P(P_i, P_j)\geq 2$ and condition A$'$-\ref{Ib:2} is satisfied for $i$.  Since these hold for each $i$, $\overline{P}$ satisfies condition A$'$ and $\text{cov}_{A'}(P)=\kappa$.
\end{proof}

A weak converse of proposition \ref{prop:2k+1_case1} also holds.

\begin{prop}\label{prop:A=A'}
	If $P$ is a graph with $r(P)>1$, $\kappa=\text{cov}_A(P)\neq \text{cov}_{A'B'}(P)$ and $\text{cov}_{A'}(P)=\text{cov}_{A}(P)$, then $A_{ucg}(C, P)=2\kappa+1$.
\end{prop}
\begin{proof}
	Assume $cov_{A'B'}(P)\neq \kappa$ but $cov_{A'}(P)=\kappa$. Let $\overline{P}=\{P_1, \ldots P_\kappa\}$ be a smallest covering with respect to condition A$'$ and  \[G=\GG(C, P, \overline{P},2)-\{x_{0,2}\}.\]\noindent  We claim $G$ is a UCG with $C=\la\ZZ(G)\ra$, $P=\la\CP(G)\ra$ and $|\II(G)|=2\kappa+1$.
	
	By proposition \ref{prop:A'B'UCG} the graph $G'=G-\{x_{0,1}\}$ is a UCG if and only if $\overline{P}$ satisfies conditions A$'$ and B$'$.    In $G'$, if  $\overline{P}$ satisfies conditions A$'$, then  $e(x_{i,1})\geq 4$ and $e(x_{i,2})\geq 4$.  This is still holds in $G$.  Furthermore, for each vertex $p\in V(P)$, $d(x_{0,1}, p)=4$ in $G$, and so $G$ is a UCG. Therefore, $A_{ucg}(C, P)\leq 2\kappa+1$, but by proposition \ref{prop:Aucg=2k} $A_{ucg}(C, P)> 2\kappa$.
\end{proof}

The case when $|D_2|=\kappa+1$ is more complicated.  To understand this case we consider a new graph $G$, and determine new conditions for when $G$ is a UCG.

For two graphs $C$ and $P$, let $\overline{P}=\{P_1, \ldots, P_k\}$ be a covering of $P$, and $\overline{Q}=\{Q_0, Q_1, P_2, \ldots, P_k\}$ be a covering such that $Q_0\cup Q_1=P_1$.

Define a graph $G=\GG'(C, P, \overline{Q})$ as follows (see figure \ref{fig:G'}):
\[V(G)=V(C) \cup V(P) \cup \{x_{i}, y_j: \, 1\leq i\leq k, 0\leq j\leq k\}\]
and  $ab$ is an edge in $G$ if and only if one of the following occurs
\begin{enumerate}[i)]
	\item $ab$ is an edge of $C$.
	\item $ab$ is an edge of $P$.
	\item $a  $ is a vertex of $C$ and $b=x_{i}$ for some $i$ with $1\leq i \leq k$.
	\item $a\in P_j$ and $b=y_{j}$ for some $j$ with $2\leq j\leq k$.
	\item $a\in Q_l$ and $b=y_{l}$ for some $l=0, 1$.
	\item $a=x_i$ and $b=y_i$ for $1\leq i\leq k$
	\item $a=x_1$ and $b=y_0$.
\end{enumerate}

\begin{figure}[!htbp]
	\[\xy
	(0,10)*\ellipse(20,4){-};
	%(-20,20)*{\bullet};(-10,20)*{\bullet};(20,20)*{\bullet};(0,20)*{\cdots};
	(0,20)*{C};
	%D_1
	(-15,35)*{\bullet};(-5,35)*{\bullet};(8,35)*{\cdots};(15,35)*{\bullet};
	(-19,35)*{x_{1}};(20,35)*{x_{k}};(0,35)*{x_2};
	%D_2
	(-15,45)*{\bullet};(-5,45)*{\bullet};(8,45)*{\cdots};(15,45)*{\bullet};(-25,45)*{\bullet};
	(-15,35)*{};(-15,45)*{}**\dir{-};(-5,35)*{};(-5,45)*{}**\dir{-};(15,35)*{};(15,45)*{}**\dir{-};(-15,35)*{};(-25,45)*{}**\dir{-};
	(-19,45)*{y_{1}};(0,45)*{y_{2}};(20,45)*{y_{k}};(-29,45)*{y_{0}};
	%Connecting to D_1
	%(-20,20)*{};(-15,35)*{}**\dir{-};(-20,20)*{};(-5,35)*{}**\dir{-};(-20,20)*{};(15,35)*{}**\dir{-};(-20,20)*{};(-35,35)*{}**\dir{-};
	(-12,25)*{};(-15,35)*{}**\dir{-};(-6,25)*{};(-15,35)*{}**\dir{-};(0,25)*{};(-15,35)*{}**\dir{-};
	(-3,25)*{};(-5,35)*{}**\dir{-};(3,25)*{};(-5,35)*{}**\dir{-};(9,25)*{};(-5,35)*{}**\dir{-};
	(6,25)*{};(15,35)*{}**\dir{-};(10,25)*{};(15,35)*{}**\dir{-};(14,25)*{};(15,35)*{}**\dir{-};
	(-20,25)*{};
	%P
	%(0,28)*\ellipse(30,6){-};
	(0, 66)*{P};
	(-26,57)*{Q_0};
	(-25,45)*{};(-28,55)*{}**\dir{-};(-25,55)*{};(-25,45)*{}**\dir{-};(-25,45)*{};(-22,55)*{}**\dir{-};
	(-15,57)*{Q_1};
	(-15,45)*{};(-18,55)*{}**\dir{-};(-15,55)*{};(-15,45)*{}**\dir{-};(-15,45)*{};(-12,55)*{}**\dir{-};
	(-4,57)*{P_2};
	(-5,45)*{};(0,55)*{}**\dir{-};(-3,55)*{};(-5,45)*{}**\dir{-};(-5,45)*{};(-6,55)*{}**\dir{-};
	(24,57)*{P_k};
	(15,45)*{};(28,55)*{}**\dir{-};(25,55)*{};(15,45)*{}**\dir{-};(15,45)*{};(22,55)*{}**\dir{-};
	(10,55)*{\cdots};
	(-13,28.3)*\ellipse(6.5,3){-};(12,28.3)*\ellipse(6.5,3){-};(-2,28.3)*\ellipse(6.5,3){-};(-7.5,28.3)*\ellipse(6.5,3){-};
	(-34,62)*{};(27,62)*{};**\frm{^)};
	\endxy\]
	\caption{$G=\GG'(C, P, \overline{Q})$}\label{fig:G'}
\end{figure}

We now determine conditions on $\overline{Q}$ for $G$ to be a UCG with $C=\la\ZZ(G)\ra$ and $P=\la\CP(G)\ra$.  By construction for all $c\in V(C)$ and $p\in V(P)$, $d(c, p)=3$, and $d(c, x)\leq 2$ for all $x\in V(G)-V(P)$.  So $G$ is a UCG with center $C$ and centred periphery $P$  if and only if for all $u\in V(G)-V(C)$, $e(u)\geq 4$.  Then there exists a vertex $v$ with $d(u,v)\geq 4$.

First consider the case $u=x_i$.  In every UCG the induced covering satisfies condition A.  This implies $v\in V(P)$ and  $e(x_{i})\geq 4$ for all $1\leq i \leq k$.  Note that, in $G$ the induced covering is $\overline{P}$, not $\overline{Q}$.

Next, consider the case $u=y_{i}$ for some $2\leq i \leq k$.  Since $v\not\in V(C)$ and $d(v, x_{j})\leq 3$ for $1\leq j\leq k$,  either $v\in V(P)$ or $v=y_j$ for $0\leq j\leq k$ and $j\neq i$.  The one of the following holds.
\begin{enumerate}[i)]
	\item  If $v\in V(P)$ then $v\not\in P_i$ and $d(P_i, v)\geq 3$.
	\item   If $v=y_{j}$, $2\leq j\leq k$ and $j\neq i$, then $d(P_i, P_j)\geq 2$.
    \item   If $v=y_{l}$ for $l=0$ or $1$, then $d(P_i, Q_l)\geq 2$.
\end{enumerate}
Now, consider the case $u=y_0$ or $y_1$.   Without loss of generality we assume $u=y_0$.  Since $v\not\in V(C)$, $v\neq x_j$ for $1\leq j\leq k$, $v\neq y_1$ and $v\not\in Q_0\cup Q_1=P_1$, the following must hold.
\begin{enumerate}[i)]
	\item  If $v\in V(P)$ then $v\not\in P_1$ and $d(Q_0, v)\geq 3$.
	\item   If $v=y_{j}$, $2\leq j\leq k$ and $j\neq i$, then $d(Q_0, P_j)\geq 2$.
\end{enumerate}

Now, consider the case $u=p\in P_i$ for some $2\leq i\leq k$.  Once again, if $G$ is a UCG, then $e(p)\geq 4$, and so there exists a $v\in V(G)$ such that $d(u, v)\geq 4$.  We know $v\not\in V(C)$.  Also, $v\not\in P_i$ since $d(p, p')\leq 2$ for all $p'\in P_i$.  Also note, that for $j\geq 2$ and $j\neq i$, if $d(u, y_j)\geq 4$ then $d(u, x_j)\geq 4$.  This implies that we do not need to determine the conditions for $d(u, y_j)\geq 4$.  Given this, one of the following must hold.

\begin{enumerate}[i)]
	\item If $v=x_{j}$ for $j\neq i$ and $2\leq j\leq k$,  then $d(u, P_j)\geq 2$.
	\item If $v=x_1$, then $d(u, Q_0)\geq 2$ and $d(u, Q_1)\geq 2$.
	\item If $v=y_l$ for $l=0$ or $1$, then $d(u, Q_l)\geq 3$.
	\item If $v=p'\in P_j-P_i$ for  $2\leq j\leq k$ and  $j\neq i$, then $d(u, p')\geq 4$ and $d(u, P_j)\geq 2$.
	\item If $v=p'\in Q_l-P_i$ for $l=0$ or $1$, then $d(p, Q_l)\geq 2$ and $d(u, p')\geq 4$.
\end{enumerate}

 The last case to consider, without loss of generality, is  $u=p\in Q_0$.  Since $d(u, v)\leq  3$ if $v\in V(C)$, $v=x_1, y_0$ or $y_1$ or $v\in Q_0$,  the following must hold.
\begin{enumerate}[i)]
	\item If $v=x_{j}$,  then $2\leq j\leq k$ and $d(u, P_j)\geq 2$.
	\item If  $v=y_{j}$ for $2\leq j\leq k$,  then $d(u, P_j)\geq 3$.
	\item If $v=p'\in P_j-Q_0$ for $2\leq j\leq k$, then $d(u, P_j)\geq 2$ and $d(u, p')\geq 4$.
	\item If $v=p'\in Q_1-Q_0$,  then  $d(u, p')\geq 4$ and $d(u, Q_1)\geq 2$.
\end{enumerate}

The above discussion is in terms of $\overline{Q}$, however the rest of the paper is in terms of $\overline{P}$.  For this reason now summarize the discussion in terms of two technical conditions on $\overline{P}$.

Let $\overline{P}=\{P_1, \ldots, P_k\}$ be a covering of a graph $P$. For a given  $\iota$ and two sets $Q_0$ and $Q_1$ such that $Q_0\cup Q_1 =P_\iota$,  let \[\overline{Q}=\overline{P} \cup \{Q_0, Q_1\}-\{P_\iota\}.\] We say $(P, \overline{P}, \overline{Q})$ satisfies
\begin{description}
	\item[Condition A$''$:] \quad%if $\overline{P}'$ satisfies condition A$'$.
		\begin{enumerate}
			\item if for each $i\neq \iota$ one of the following holds \label{A'':1}
			\begin{enumerate}
				\item  there is a $p\not\in P_i$ satisfying $d(P_i, p)\geq 3$, or \label{A'':1a}
				\item there is a $j\neq \iota$ satisfying $d(P_i, P_j)\geq 2$, or \label{A'':1b}
				\item there is an $l= 0$ or $1$ such that $d(P_i, Q_l)\geq 2$, \label{A'':1c}
			\end{enumerate}	
			\item and if for all $l= 0, 1$ either
			\begin{enumerate}\label{A'':2}
				\item  there is a $p\not\in P_\iota$ satisfying $d(Q_l, p)\geq 3$, or \label{A'':2a}
				\item there is a $j\neq \iota$ satisfying $d(Q_l, P_j)\geq 2$. \label{A'':2b}
			\end{enumerate}
		\end{enumerate}
	\item[Condition B$''$:] \quad
	\begin{enumerate}
		\item if for each $p\in P_i$, $i\neq \iota$ one of the following holds \label{B'':1}
		\begin{enumerate}
			\item  there is a $j\neq \iota$ satisfying $d(p, P_j)\geq 2$, or \label{B'':1a}
			\item $d(p, Q_0)\geq 2$ and $d(p, Q_1)\geq 2$, or \label{B'':1b}
			\item there is an $l=0$ or $1$ such that $d(p, Q_l)\geq 3$, or \label{B'':1c}			
			 \item there  $l=0$ or $1$ such that $d(p, Q_l)\geq 2$ and a $p'\in Q_l$ so that $d(p, p')\geq 4$. \label{B'':1d}
		\end{enumerate}
		\item and if for each $l=0, 1$ and each $p\in Q_l$, either \label{B'':2}
		\begin{enumerate}
			\item  there is a $j\neq \iota$ satisfying $d(p, P_j)\geq 2$, or \label{B'':2a}
			\item  there exists $p'\in P_\iota - Q_l$ such that $d(p, p')\geq 4$ and $d(p, Q_{l'})\geq 2$, where $l'=0$ or $1$ but $l'\neq l$. \label{B'':2b}
		\end{enumerate}
	\end{enumerate}
\end{description}

When there exists an $\iota$, $Q_0$ and $Q_1$ such that $(P, \overline{P}, \overline{Q})$ satisfies condition A$''$ (\textit{resp.} condition B$''$), we say $(P, \overline{P})$ or simply $\overline{P}$ satisfies condition A$''$ (condition B$''$).  Without loss of generality we may renumber $\overline{P}$ so that $\iota=1$.   Note that A$'$ and B$'$ imply A$''$ and B$''$ by taking $\iota=1$  and letting $Q_0=P_1$ and $Q_1=\emptyset$. However, condition A$''$ does not imply condition A.  So let $\text{cov}_{AA''B''}(P)$ be the smallest size of the covering $\overline{P}$ of $P$ satisfying conditions A, A$''$ and B$''$.

Similar to the discussion of conditions A$'$ and B$'$, the arguments used to determine conditions A$''$ and B$''$ from the graph $\GG(C, P, \overline{Q})$ are  bi-directional.  We summarize the discussion in the following proposition.

\begin{prop}\label{prop:A''B''UCG}
	For graphs $C$ and $P$, and a triple $(P, \overline{P}, \overline{Q})$,  $G=\GG'(C, P, \overline{Q})$ is a UCG if and only if $\overline{P}$ satisfies condition A and $(P,\overline{P}, \overline{Q})$ satisfies conditions A$''$ and B$''$.
\end{prop}

We are now ready to prove analogues to propositions  \ref{prop:2k+1_case1} and \ref{prop:A=A'}.

\begin{prop}\label{prop:2k+1_case2}
	Assume $A_{ucg}(C, P)= 2\kappa+1$, and let $H$ be a UCG with $\la\ZZ(H)\ra= C$, $\la\CP(H)\ra=P$, and $|\II(H)|=2\kappa+1$ where $\kappa=\text{cov}_A(P)$.  Furthermore, assume $|D_2|=\kappa+1$.  Then $\text{cov}_{AA''B''}(P)=\kappa$.
\end{prop}

\begin{proof}
Let $H$ be a UCG with $C=\la\ZZ(H)\ra$, $P=\la\CP(H)\ra$, $|\II(H)|=2\kappa+1$ and $|D_2|=\kappa+1$.  Since $|D_1|=\kappa$ and $\kappa\leq |D_1''|\leq |D_1|$, it follows that $|D_1|=|D_1''|$.    We prove this proposition by studying the structure of a spanning subgraph.

Let $D_1=\{x_1, \ldots, x_\kappa\}$ and $\{\tilde p_1, \ldots, \tilde p_\kappa\}$ be a set of vertices associated to the induced cover $\overline{P}$.  By proposition \ref{prop:uniquepaths} there exists an enumeration $\{y_0, \ldots, y_{\kappa}\}$ of $D_2$ such that the vertex $y_j$ is adjacent to $x_j$ and $\tilde p_j$ for each $j$, $1\leq j\leq \kappa$.  We may also assume $y_0$ is adjacent to $x_1$.

We now define a different cover of $P$. For each $i$, $0 \leq i\leq \kappa$,   let
\[Q_i=\{p\in V(P): p \text{ is adjacent to } y_i\}.\]
Let $P_1'=Q_0\cup Q_1$ and $P_i'=Q_i$ for $2\leq i \leq \kappa$, $\overline{P}'=\{P_1', \ldots, P_\kappa'\}$ and $\overline{Q}=\{Q_0, \ldots, Q_\kappa\}$.
Since $|D_1|=|D_1''|$  every vertex of $C$ is adjacent to vertex in $D_1$ by lemma \ref{lem:C-D_1}.  Therefore, $G=\GG'(C, P, \overline{Q})$ is isomorphic to a spanning subgraph of $H$, and is a UCG by lemma \ref{lem:spanUCG}.  By proposition \ref{prop:A''B''UCG} $\overline{P}'$ satisfies conditions A, A$''$ and B$''$ which means
\[\kappa\leq \text{cov}_{AA''B''}(P)\leq |\overline{P}'|= \kappa.\]
We conclude $\text{cov}_{AA''B''}(P)=\kappa$.

\end{proof}

\begin{prop}\label{prop:A=A''B''}
	If $P$ is a graph with $r(P)>1$, $\kappa=\text{cov}_A(P)\neq \text{cov}_{A'B'}(P)$ and $\text{cov}_{AA''B''}(P)=\text{cov}_{A}(P)$, then $A_{ucg}(C, P)=2\kappa+1$.
\end{prop}
\begin{proof}
	Assume $cov_{A'B'}(P)\neq \kappa$ but $cov_{AA''B''}(P)=\kappa$. By proposition \ref{prop:Aucg=2k} $A_{ucg}(C, P)> 2\kappa$, so we need to show $A_{ucg}(C, P)\leq 2\kappa+1$.  Let $\overline{P}$ be a smallest covering of $P$  with respect to condition A such that there is a refined cover $\overline{Q}$ where the pair $(\overline{P}, \overline{Q})$ is smallest with respect to conditions A$''$ and B$''$.  Without loss of generality we may assume $\iota=1$. By proposition \ref{prop:A''B''UCG}  $G=\GG(C, P, \overline{Q})$ is a UCG with $C=\la\ZZ(G)\ra$, $P=\la\CP(G)\ra$ and $|\II(G)|=2\kappa+1$, which implies $A_{ucg}(C, P)\leq 2\kappa+1$.
\end{proof}

The following theorem summarizes propositions \ref{prop:Cbound}, \ref{prop:Aucg=2k}, \ref{prop:2k+1_case1}, \ref{prop:A=A'}, \ref{prop:2k+1_case2}, and \ref{prop:A=A''B''}.

\begin{thm}\label{thm:CRange}
	If $P$ is a graph with $r(P)>1$ and $\text{cov}_A(P)=\kappa$, then following holds:
	\begin{enumerate}
		\item $A_{ucg}(C, P)=2\kappa$ if and only if $\text{cov}_{A'B'}(P)=\kappa$.
		\item $A_{ucg}(C, P)=2\kappa+1$ if and only if  $\kappa\neq \text{cov}_{A'B'}(P)$ and either
		\begin{enumerate}
			\item $\text{cov}_{A'}(P)=\kappa$ or
			\item $\text{cov}_{AA''B''}(P)=\kappa$.
		\end{enumerate}
		\item $A_{ucg}(C, P)=2\kappa+2$ otherwise.
	\end{enumerate}
\end{thm}

\section{Coverings}\label{sec:cov}
In this section we determine when $\kappa=\text{cov}_A(P)$ is the size of smallest coverings with respect to the other conditions described in sections \ref{sec:C=Kn} and \ref{sec:C<>Kn}.  To do this we introduce one last set of notation.  For a graph $G$, a vertex $V$ in $G$ and $s\in \Nb$ let
\[N_s[v]=\{x\in V(G)| d(v, x)\leq s\}\]
\noindent be the closed $s$ neighborhood of $v$.  When $s=1$ we simply let $N_1[x]=N[x]$.

Proposition \ref{prop:Knbound} and theorem \ref{thm:Knrange} determine $A_{ucg}(K_n, P)$ up to knowing when  $\text{cov}_A(P)=\text{cov}_{AB}(P)$.  We now determine conditions for a graph $P$ to satisfy $\text{cov}_A(P)=\text{cov}_{AB}(P)$.
\begin{prop}\label{d>3}
	If $P$ is a graph with $\text{diam}(P)\geq 3$, then $\text{cov}_A(P)=2$.
\end{prop}
\begin{proof}
	Suppose $x$ is a vertex in $P$  satisfying  $e(x)\geq 3$.  Let $P_1=N[x]$  and $P_2=V(P) - P_1$.  Then $d(x,P_2)\geq 2$.  Since $e(x)\geq 3$, $P_2$ contains a vertex $y$ satisfying $d(x, y) \geq 3$, and so $d(y, P_1)\geq 2$.  Thus, $\{P_1, P_2\}$ is a covering of $P$ satisfying condition A.
\end{proof}

\begin{prop}\label{d>4r>3}
	If $P$ is a graph satisfying $\text{diam}(P)\geq 4$ and $r(P)\geq 3$, then $\text{cov}_{AB}(P)=2$.
\end{prop}
\begin{proof}
	Let $x$ and $y$ be vertices in  $P$ satisfying $d(x, y)=4$.  We construct $P_1$ and $P_2$ recursively.  Initialize $P_1:=N[x]$ and $P_2:=N[y]$.  Note that $d(x, N[y]) = d(y, N[x])=3$.   For a vertex $z$ in $V(P) -(P_1\cup P_2)$, update $P_1$ and $P_2$ as follows.
	\begin{enumerate}[i)]
		\item If there is a vertex $p\in P_2$ satisfying  $d(z, p)\geq 3$, then let $P_1:=P_1\cup \{z\}$.
		\item Else if there is a vertex $p\in P_1$ satisfying $d(z, p)\geq 3$, then let $P_2:=P_2\cup \{z\}$.
		\item Else there is a vertex $p\in V(P) -  (P_1\cup P_2)$ satisfying $d(z, p)=3$, since $r(P)=3$.  Let  $P_1:=P_1\cup \{z\}$ and $P_2:=P_2\cup \{p\}$.
	\end{enumerate}
	Continue until all vertices of $P$ have been accounted for. By construction $\{P_1, P_2\}$ is a covering of $P$  satisfying condition B.   Since $N[x]\subset P_1$,  $d(x, P_2)\geq 2$ and similarly $d(y, P_1)\geq 2$. Thus $\{P_1, P_2\}$ satisfies condition A.
\end{proof}

\begin{prop}\label{r=2}
	If $P$ is a graph with $r(P)=2$, then $\text{cov}_{AB}(P)\neq 2$.
\end{prop}
\begin{proof}
	Suppose there is a covering $\{P_1, P_2\}$ of $P$ satisfying conditions A and B.  Assume $P_1$ contains a vertex  $z$ with $e(z)= 2$.   By condition B, $d(z, P_2)=2$ and so $N(z)\cap P_2=\emptyset$.  If  $p$ is a vertex in $P_2$, then $d(z, p)=2$ and so $d(P_1, p)=1$, a contradiction to condition A.
	
	%Now let $z'\in \ZZ(M)$.  We will show $z'\in M_1$.   If $z'\in N^1(z)$ then $z'\in M_1$.   So assume $z'\not\in M^1(z)$ and $z'\in M_2$.  Then $N^1(z')\cap M_1=\emptyset$, but since $d(z,z')=2$ there is a $y\in N^1(z)\cap N^1(z')$, so $y\not\in M_1\cup M_2$.  We conclude $\ZZ(M)\cup N^1(\ZZ(M))\subset M_1$.
	
\end{proof}

\begin{prop}\label{d=r=2}
	For every $\alpha\in\Nb$, there exists a graph $P$ with $r(P)=\text{diam}(P)=2$ and $\text{cov}_A(P)=2\alpha$.
\end{prop}
\begin{proof}
	Construct $P^\alpha=(V^\alpha,E^\alpha)$ as follows.  Let \[V^\alpha=\{e_i, f_i: \, 1\leq i \leq \alpha\}\] \noindent and \[E^\alpha=\{e_ie_j, f_i f_j, e_if_j:   1\leq i,j\leq \alpha, i\neq j\}.\]
	 \noindent Then $r(P^\alpha)=\text{diam}(P^\alpha)=2$.
	
	We now show $\text{cov}_A(P^\alpha)=2\alpha$.   By lemma \ref{lem:covexist} there exists a covering  \[\{P^\alpha_1, P^\alpha_2, \ldots, P^\alpha_k\}\] \noindent of $P^\alpha$ satisfying condition A.  Suppose $e_1\in P^\alpha_1$.     If $P^\alpha_1$ contains a vertex other than $e_1$,  then $d(P^\alpha_1, x)=1$ for all $x\in V^\alpha-P^\alpha_1$, a contradiction to condition A. Therefore  $P^\alpha_1=\{e_1\}$.  Similarly $|P^\alpha_i|=1$ for all $i=1,2, \ldots, k$ and $\text{cov}_A(P^\alpha)=2\alpha$.
\end{proof}

Propositions \ref{d>3} through \ref{r=2} determine when $\text{cov}_A(P)=\text{cov}_{AB}(P)$ for all graphs $P$ with $r(P)>1$ except those with $r(P)=\text{diam}(P)=2$ and those with $r(P)=\text{diam}(P)=3$.  Proposition \ref{d=r=2} gives insight into richness of the case $r(P)=\text{diam}(P)=2$.  However, we do not have any definitive results for $r(P)=\text{diam}(P)=3$.  This is discussed further in section \ref{sec:Append}.

For a non-complete graph $C$,  theorem \ref{thm:CRange} relates $A_{ucg}(C,P)$ to $\text{cov}_{A'B'}(P)$, $\text{cov}_{A'}(P)$, and $\text{cov}_{AA''B''}(P)$.  Because of proposition \ref{d=r=2}, we only consider $P$ with $\text{diam}(P)>2$.  By proposition \ref{d>3} we need to understand when the smallest coverings are of size $2$.

\begin{prop}\label{d>3'}
	$P$ is a graph with $\text{cov}_{A'B'}(P)=2$  if and only if $P$ is disconnected.
\end{prop}
\begin{proof}
Assume $P$ is disconnected.  Let $P_1$ be the vertices of a connected component of $P$ and let $P_2=V(P)-P_1$.   Since, for all $u\in P_1$ and $v\in P_2$, $d(u, v)=\infty$ it follows that $\{P_1, P_2\}$ satisfies conditions A$'$ and B$'$.

Next, assume $P$ is connected but $\overline{P}=\{P_1, P_2\}$ is a covering.  Since $P$ is connected, $d(P_1, P_2)=1$  and there is a $p_1\in P_1$ and $p_2\in P_2$ with $d(p_1, p_2)=1$. Therefore $d(p_1, P_2)=1$ and $d(p_2, P_1)=1$ and condition B$'$ fails, and so $\text{cov}_{A'B'}(P)>2$.
\end{proof}

\begin{prop}\label{d>5'}
	For a graph $P$, $\text{cov}_{A'}(P)=2$ if and only if $\text{diam}(P)\geq 5$.
\end{prop}
\begin{proof}
	Assume $\text{diam}(P)\geq 5$.  Then there exist $u, v\in V(P)$ with $d(u, v)\geq 5$.  Let $P_1=\{p\in V(P): d(u, p)\leq 2\}$ and $P_2=V(P)-P_1$.  Then $d(u, P_2)\geq 3$.  Also $d(v, P_1)\geq 3$ because if there is a $p\in P_1$ with $d(v, p)\leq 2 $, then $d(u, v)\leq d(u, p)+d(p, v)\leq 4$, a contradiction. So  $\{P_1, P_2\}$ satisfies  condition A$'$.
	
	Next, assume there is a covering $\{P_1, P_2\}$ of $P$ satisfying condition A$'$.  Then there is a $u\in P_1$ such that $d(u, P_2)\geq 3$.  Then $N_2[u] \subset P_1$ and $N_2[u]\cap P_2=\emptyset$.  Similarly, $P_2$ contains a vertex $v$ such that $d(v, P_1)\geq 3$, and so  $N_2[v]\subset P_1$ and $N_2[v]\cap P_1=\emptyset$.  Therefore $d(u,v)\geq 5$.
\end{proof}

\begin{prop}\label{d=3''}
	If $P$ is a graph with $\text{diam}(P)\leq 3$, then $\text{cov}_{AA''B''}(P)\neq 2$.
\end{prop}
\begin{proof}
		Let $P$ be a  graph with $\text{diam(P)}\leq 3$.  Suppose there is a covering $\overline{P}=\{P_1, P_2\}$ of $P$ satisfying condition A such that for $\iota=1$ the triple  $(P, \overline{P}, \overline{Q})$ satisfies conditions A$''$ and B$''$, where $\overline{Q}=\{Q_0, Q_1, P_2\}$ with  $P_1=Q_0\cup Q_1$.  For each $p\in P_1$,  $d(p, P_2)\geq 2$ by condition B$''$-2 because $\text{diam}(P)\leq 3$.  This implies $d(P_1, P_2)\geq 2$.  However,  $d(P_1, P_2)\leq 1$ since $P$ is connected, a contradiction.
\end{proof}

\begin{prop}\label{r=2''}
	If $P$ is a graph with $r(P)=2 $, then $\text{cov}_{AA''B''}(P)\neq 2$.
\end{prop}
\begin{proof}
	Let $P$ be a graph with $r(P)=2$.  Suppose there is a covering $\overline{P}=\{P_1, P_2\}$ of $P$ satisfying condition A such that for $\iota=1$ the triple  $(P, \overline{P}, \overline{Q})$ satisfies conditions A$''$ and B$''$, where $\overline{Q}=\{Q_0, Q_1, P_2\}$ with   $P_1=Q_0\cup Q_1$.
	
	For a central vertex $c\in V(P)$ either  $c\in P_1$ or $c\in P_2$.
	
	Suppose $c\in P_1$.  Without loss of generality assume $c\in Q_0$.  By condition B$''$-2, $d(c, P_2)\geq 2$.  In fact, $d(c, P_2)=2$ because $c$ is central, and so there exists a $p\in P_2$ with $d(c, p)=2$. Let $c-x-p$ be a dmpath.  Then $x\in P_1$, $e(x)\leq 3$ and $d(x, P_2)=1$.  Therefore, $\{P_1, P_2\}$ does not satisfy condition $B''$-2 for $x$, a contradiction.
	
	Next, suppose $c\in P_2$.  By the hypotheses $c$ does not satisfy B$''$-1a, B$''$-1c, and B$''$-1d.  Therefore, $d(c, Q_0)\geq 2$ and $d(c, Q_1)\geq 2$.  Since $c$ is a central vertex,  $d(c, Q_0)= 2$ and $d(c, Q_1)= 2$.  Then there exist $q_0\in Q_0$ and $q_1\in Q_1$, satisfying  $d(c, q_0)= 2$ and $d(c, q_1)= 2$.  Let $c-x-q_0$ and $c-y-q_1$ be paths.  Then both $x$ and $y$ are in $P_2$ and $d(P_2, Q_i)\leq 1$ for $i=0, 1$.  Therefore $\{P_1, P_2\}$ does not satisfy condition A$''$-1.
	
\end{proof}
To fully understand $A_{ucg}(C,P)$ there are still two cases left to consider,  when $\text{diam}(P)=4$ and $r(P)=3$, and  $\text{diam}(P)=4$ and $r(P)=4$.  In these cases $\text{cov}_{A'B'}(P)\neq 2$ and $\text{cov}_{A'}(P)\neq 2$.   When $\text{diam}(P)=r(P)=4$, $\text{cov}_{AA''B''}(P)=2$ depends on $P$.  To show that there exist graphs $P$ with  $\text{cov}_{AA''B''}(P)\neq 2$, we introduce the following lemma.

%After proposition \ref{r=2''}

\begin{lemma}\label{lem:2balls}
	If $P$ is a graph with $\text{diam}(P)= 4$ and $\text{cov}_{AA''B''}(P)= 2$, then there exist three vertices $x_1, x_2, x_3$ such that \[N_2[x_1]\cap N_2[x_2]\cap N_2[x_3]=\emptyset.\]
\end{lemma}

\begin{proof}
		Suppose there is a covering $\overline{P}=\{P_1, P_2\}$ of $P$ satisfying condition A such that for $\iota=1$ the triple  $(P, \overline{P}, \overline{Q})$ satisfies conditions A$''$ and B$''$, where $\overline{Q}=\{Q_0, Q_1, P_2\}$ with  $P_1=Q_0\cup Q_1$.
		
		Since $P$ is connected  without loss of generality we may assume that $d(P_2, Q_1)\leq 1$.  Then A$''$-\ref{A'':2b} is not satisfied for $l=1$ and so, from A$''$-\ref{A'':2a}, there exists a $p_1\in P_2$ such that $d(Q_1, p_1)\geq 3$.  For $l=0$ there are two cases for  condition A$''$-\ref{A'':2}, either part \ref{A'':2a} is met or part \ref{A'':2b} is met.

		For $l=0$, either there exists a $p_0\in P_2$ such that $d(Q_0, p_0)\geq 3$ or $d(P_2, Q_0)\geq 2$ by condition A$''$-\ref{A'':2b}.
		
		Suppose  $d(P_2, Q_0)\geq 2$, by connectivity of $P$ we know $d(P_2, Q_1)\leq 1$ and $d(Q_0, Q_1)\leq 1$. Since $d(p_1, Q_1)\geq 3$, $N_2[p_1]\cap Q_1=\emptyset$ and so \[N_2[p_2]\subset (P_2\cup Q_0)-Q_1.\]
		Since $p_1\in P_2$ and $d(Q_0, P_2)\geq 2$, $N_1[p_1]\cap Q_0=\emptyset$ and  $N_1[p_1]\subset P_2$. From
		\[2\leq d(Q_0, P_2)\leq d(Q_0, N_1[p_1]),\]
		it follows that $N_2[p_1]\subset P_2$.  Therefore, $d(N_2[p_1], Q_0)\geq 2$ and $d(p_1, Q_0)\geq 4$.  Since $\text{diam}(P)=4$, $d(p_1, Q_0)= 4$.
		
		 Let  $p_1-x-y-z-q_0$ be a dmpath  for some $q_0\in Q_0$.  Then $z\not\in Q_0$ since $d(z, p_1)=3$, and $z\not\in P_2$ since $d(q_0, z)=1$.  Hence $z\in Q_1$.  Furthermore, $y\in N_2[p_1]\subset P_2$,  so $d(z, P_2)\leq 1$ and $d(z, Q_0)\leq 1$.  This implies condition B$''$-\ref{B'':2} is not met for $z\in Q_1$, a contradiction.  Therefore $d(P_2, Q_0)\leq 1$.
		
		  Then, for $l=0$ condition A$''$-\ref{A'':2a} is satisfied and there is a $p_0\in P_2$ such that  $d(Q_0, p_0)\geq 3$.  Therefore, $N_2[p_0]\cap Q_0=\emptyset$. Since $N_2[p_1]\cap Q_1=\emptyset$, it follows that $N_2[p_0]\cap N_2 [p_1]\subset P_2$.
		
		  From condition A$''$-\ref{A'':1}  there  exists a $q\in Q_0$ such that $d(q, P_2)\geq 3$.  Then $N_2[q]\cap P_2=\emptyset$ and
		  \[N_2[p_0]\cap N_2[p_1]\cap  N_2[q]=\emptyset.\]		
	\end{proof}
\begin{prop}\label{r4d4<>2}
	There exists a graph $P$ with $r(P)=\text{diam}(P)= 4$ and $\text{cov}_{AA''B''}(P)\neq 2$.
\end{prop}

\begin{proof}
	Let $P$ be the graph of a hexagonal prism as in figure \ref{fig:hex*2}.   Then $r(P)=\text{diam}(P)=4$. If $\text{cov}_{AA''B''}(P)= 2$ then by lemma \ref{lem:2balls} there exist vertices $x_1, x_2, x_3$ such that
	\[N_2[x_1]\cap N_2[x_2]\cap N_2[x_3]=\emptyset,\]
	or equivalently,
	\[N_2[x_1]^c\cup N_2[x_2]^c\cup N_2[x_3]^c = V(P).\]
	For any vertex $x$ of $P$, the compliment of $N_2[x]$ is a star with three pendant vertices as in figure \ref{fig:star}.  Because the prism has twelve vertices, a set of   three stars would cover the prism with no overlap.  However, one can check this is not possible and so $\text{cov}_{AA''B''}(P)\neq 2$

	\begin{figure}[h!]
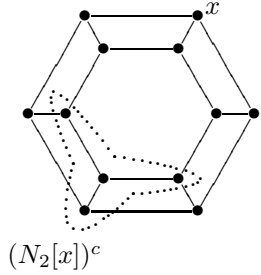

		\centering
		%\begin{subfigure}{0.5\textwidth}
				\[\xy
				(10,0)*{\bullet}="11";(5,8.66)*{\bullet}="21";(-5,8.66)*{\bullet}="31";(-10,0)*{\bullet}="41";
				(5,-8.66)*{\bullet}="61";(-5,-8.66)*{\bullet}="51";
				(15,0)*{\bullet}="12";(7.5,13)*{\bullet}="22";(-7.5,13)*{\bullet}="32";(-15,0)*{\bullet}="42";
				(7.5,-13)*{\bullet}="62";(-7.5,-13)*{\bullet}="52";
				"11";"21"**\dir{-};"21";"31"**\dir{-};"31";"41"**\dir{-};"41";"51"**\dir{-};
				"51";"61"**\dir{-};"61";"11"**\dir{-};	"12";"22"**\dir{-};"22";"32"**\dir{-};"32";"42"**\dir{-};"42";"52"**\dir{-};
				"52";"62"**\dir{-};"62";"12"**\dir{-};
				"11";"12"**\dir{-};"21";"22"**\dir{-};"31";"32"**\dir{-};"41";"42"**\dir{-};
				"51";"52"**\dir{-};"61";"62"**\dir{-};
				"22"+(2,1)*{x};
				"51"+(-4, 2)*{};"51"+(1.5,3)*{}**\crv{~*=<3pt>{.} (-17, 12)};
				"51"+(4, -2)*{};"51"+(1.5,3)*{}**\crv{~*=<3pt>{.} (18, -9)};
				"51"+(4, -2)*{};"51"+(-4, 2)*{}**\crv{~*=<3pt>{.} (-13, -22)};
				"52"+(-4,-6)*{(N_2[x])^c};
				\endxy\]
			%\subcaption{A hexagonal prism.}\label{fig:hex*2}
		%\end{subfigure}
		%\begin{subfigure}{0.4\textwidth}
		%			\[\xy
		%			(0,4)*{\bullet};(0,0)*{\bullet}**\dir{-};
		%			(3.4,-2)*{\bullet};(0,0)*{\bullet}**\dir{-};
		%			(-3.4,-2)*{\bullet};(0,0)*{\bullet}**\dir{-};
		%			(0,15)*{};(0,-13)*{};
		%			\endxy\]
					
		%	\subcaption{The compliment of $N_2[x]$.}\label{fig:star}
		%\end{subfigure}
		\caption{A hexagonal prism and the compliment of a $2$-neigborhood.}\label{fig:UFOn}\label{fig:hex*2}\label{fig:star}
	\end{figure}
\end{proof}

\begin{prop}\label{r4d4=2}
	There exists a graph $P$ with $r(P)=\text{diam}(P)= 4$ and $\text{cov}_{AA''B''}(P)= 2$.
\end{prop}

\begin{proof}
	Let $P$ be the graph of a heptagonal prism as in figure \ref{fig:hept*2}.  Then $r(P)=\text{diam}(P)=4$.  Consider the covering $\overline{P}=\{P_1, P_2\}$, where $P_1=Q_0\cup Q_1$ as in figure \ref{fig:hept*2}.  Here vertices of $Q_0$ are represented by the open circles, $Q_1$ by the open squares, and $P_2$ by the filled circles.  One can verify this covering satisfies conditions A, A$''$ and B$''$, and so $\text{cov}_{AA''B''}(P)= 2$.
	\begin{figure}[htbp!]
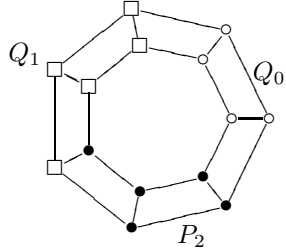

	\[\xy
	(10,0)*{\circ}="11";(6.2,7.8)*{\circ}="21";(-2.2,9.75)*{\Box}="31";(-9,4.3)*{\Box}="41";
	(6.2,-7.8)*{\bullet}="71";(-2.2, -9.75)*{\bullet}="61";(-9,-4.3)*{\bullet}="51";
	(15,0)*{\circ}="12";(9.3,11.7)*{\circ}="22";(-3.3,14.6)*{\Box}="32";(-13.5,6.5)*{\Box}="42";
	(9.3,-11.7)*{\bullet}="72";(-3.3, -14.6)*{\bullet}="62";(-13.5,-6.5)*{\Box}="52";
	"11";"21"**\dir{-};"21";"31"**\dir{-};"31";"41"**\dir{-};"41";"51"**\dir{-};
	"51";"61"**\dir{-};"61";"71"**\dir{-};"71";"11"**\dir{-};
	"12";"22"**\dir{-};"22";"32"**\dir{-};"32";"42"**\dir{-};"42";"52"**\dir{-};
	"52";"62"**\dir{-};"62";"72"**\dir{-};"72";"12"**\dir{-};
	"11";"12"**\dir{-};"21";"22"**\dir{-};"31";"32"**\dir{-};"41";"42"**\dir{-};
	"51";"52"**\dir{-};"61";"62"**\dir{-};"71";"72"**\dir{-};
	"12"+(0,6)*{Q_0};
	"42"+(-4,2)*{Q_1};
	"62"+(8,-1)*{P_2};
	\endxy\]
	\caption{A graph $P$ with $r(P)=\text{diam}(P)=4$ and $\text{cov}_{AA''B''}(P)=2$}\label{fig:hept*2}
	\end{figure}
\end{proof}

\section{Appendage Numbers}\label{sec:Append}
In this section we determine $A_{ucg}(C, P)$ for most pairs of graphs $(C,P)$ based on the structure of $C$.  Conjectures are also given for the two remaining cases.
\begin{thm}\label{prop:Appendv}
	Let $C=\{v\}$ and $P$ be any graph with $r(P)\geq 2$.  Then $A_{ucg}(C,P)=0$.
\end{thm}
\begin{proof}
	Let $H$ be the graph with vertex set
	\[V(H)=V(P)\cup\{v\}\]
	and edge set
	\[E(H)=E(P)\cup \{vp|p\in V(P)\}.\]
	We claim $H$ is a UCG with $\langle\ZZ(H)\rangle=C$, $\langle\CP(H)\rangle=P$ and $|\II(H)|=0$.
	
	Since $v$ is adjacent to all other vertices in $H$, $e(v)=1$.  If $p$ is a vertex in $P$
	there is a $p'\in V(P)$ satisfying $d_P(p,p')=2$. Since $p$ and  $p'$ are not adjacent in $P$, they are not adjacent in $H$ and $e(p)\geq 2$. This implies  $\la\ZZ(H)\ra=\{v\}$ and $H$ is UCG with $\la\CP(H)\ra=P$.
\end{proof}

\begin{thm}\label{thm:KnAppend}
	When  $n\geq 2$
	\[A_{ucg}(K_n ,P)=\begin{cases}
	2 & \text{if } \text{diam}(P)\geq 4\text{ and } r(P)\geq 3\\
	3 & \text{if } \text{diam}(P)\geq 3\text{ and } r(P)=2\\
	\end{cases}\]
	Furthermore, for all $t\in \Nb$ there is a graph $P$ with $\text{diam}(P)=r(P)=2$ and $A_{ucg}(K_n, P)\geq t$.
\end{thm}
\begin{proof}
	This follows directly from proposition \ref{prop:Knbound}, theorem \ref{thm:Knrange} and propositions \ref{d>3}, \ref{d>4r>3}, \ref{r=2}, and \ref{d=r=2}.
\end{proof}

\begin{conj}
	If $\text{diam}(P)=r(P)=3$, then $A_{ucg}(K_n, P)=2$.
\end{conj}

When $\text{diam}(P)=r(P)=3$,  $A_{ucg}(C, P)=2$ or $3$.  However, all our examples show $A_{ucg}(K_n, P)=2$, but we have not been able to prove this is always true.

\begin{thm}\label{thm:CAppend}
	If $C$ is a graph with $\text{diam}(C)>1$ then
	\[A_{ucg}(C ,P)=\begin{cases}
	4 & \text{if } \text{diam}(P)=\infty\\
	5  & \text{if } 5\leq \text{diam}(P) < \infty\\
	5 \text{ or } 6  & \text{if } \text{diam}(P)=4 \text{ and } r(P)= 4 \\
	6 & \text{if } \text{diam}(P)=4 \text{ and } r(P)=2\\
	6 & \text{if }  \text{diam}(P) =3\\
	\end{cases}\]
	Furthermore, for all $t\in \Zb$ there is a graph $P$ with $\text{diam}(P)=r(P)=2$ and $A_{ucg}(C, P)\geq t$.
\end{thm}
\begin{proof}
	This follows directly from theorem \ref{thm:CRange}, proposition  \ref{d>3} and propositions  \ref{d=r=2}, through \ref{r4d4=2}.
\end{proof}

The only case not accounted for is when $\text{diam}(P)=4$ and $r(P)=3$.   In this case we know $\text{cov}_{A'B'}(P)\neq 2$, so $A_{ucg}(C, P)= 5$ or $6$.  We also know $\text{cov}_{A'}(P)\neq 2$.  Therefore $A_{ucg}(C, P)= 5$ if and only if $\text{cov}_{AA''B''}(P)= 2$.  However, we have not found this to be the case for any such $P$.  We also have been unable to show it is impossible, so we are left with the following conjecture.

\begin{conj}
	If $C$ is a graph with $\text{diam}(C)>1$  and $P$ is a graph with $\text{diam}(P)=4$ and $r(P)=3$, then $A_{ucg}(C, P)=6$.
\end{conj}

The  $\text{diam}(P)=r(P)=4$ case also warrants further discussion.  Propositions \ref{r4d4<>2} and \ref{r4d4=2} show there are examples of $P$ when $A_{ucg}(C, P)=6$ and with $A_{ucg}(C, P)=5$. It will be necessary to find another metric invariant other than diameter and radius to refine the results of this case.  At this point we are unsure what a suitable invariant may be.

Finally, independent of $C$, there is a major difference between possible appendage numbers when  $\text{diam}(P)>2$ and when  $\text{diam}(P)=2$. When $\text{diam}(P)>2$, theorems \ref{thm:KnAppend} and \ref{thm:CAppend} show there are only finitely many possible appendage numbers, and which are independent of the size of $V(P)$.  On the other hand, for $\text{diam}(P)=r(P)=2$ the graph $P^\alpha$ in proposition \ref{d=r=2} gives \[A_{ucg}(K_n, P^\alpha) = |V(P^\alpha)| = 2\alpha.\]
This may suggest that appendage numbers are related to $|V(P)|$ when $\text{diam}(P)=r(P)=2$, however the following proposition shows this is not the case.

%	\[Q(P)=  \frac{A_{UCG}(K_n, P)}{|P|}\]
%	Since $\frac{A_{UCG}(K_n, P)}\leq \text{cov}_A(P)\leq |P|$ we know $Q(P)\leq 1$, but as proposition \ref{Q(P)} will show we can have arbitrary $Q(P)$ with arbitrarily large appendage numbers.

\begin{prop}\label{prop:2a+b}
	For every  $\alpha, \beta\in \Nb$ there is a graph $P$ such that $V(P)=2\alpha+\beta $ and $A_{ucg}(K_n, P)=2\alpha$.
\end{prop}

\begin{proof}
	In this proof we modify the construction of $P^\alpha$ from proposition \ref{d=r=2}.  For $\alpha, \beta\in \Nb$ define $P^{\alpha, \beta}=(V^{\alpha,\beta}, E^{\alpha,\beta})$ as follows.  Let \[V^{\alpha, \beta}=\{e_i, f_i, g_j:1\leq i\leq \alpha, 1\leq j\leq \beta\}\]
	and  \[E^{\alpha, \beta}=\{ e_ie_j, f_if_j, e_if_j, e_kg_l, f_kg_l : 1\leq i, j\leq \alpha, i\neq j,  2\leq k\leq \alpha,  1\leq l\leq \beta \}.\]
	That is, $P^{\alpha, \beta}$ is the graph from proposition \ref{d=r=2} with $\beta$ new vertices, $g_i's$, that are adjacent to every vertex except themselves, $e_1$ and $f_1$. Observe that $\text{diam}(P^{\alpha, \beta})=r(P^{\alpha, \beta})=2$ if $\alpha \geq 2$.
	
	We now show $\text{cov}_{A}(P^{\alpha, \beta})=\text{cov}_{AB}(P^{\alpha,\beta})=2\alpha$.  Let $\overline{P^{\alpha, \beta}}=\{P_1, \ldots, P_k\}$ be a covering of $P^{\alpha, \beta}$ satisfying condition A.  For $2\leq i \leq \alpha$,  $\{e_i\}, \{f_i\}\in \overline{P^{\alpha, \beta}}$  as in proposition \ref{d=r=2}.  If \[\{e_1, f_1, g_1, \ldots, g_{\beta}\}\in \overline{P^{\alpha, \beta}},\]  then $\overline{P^{\alpha, \beta}}$ fails to satisfy condition A. Therefore $e_1$, $f_1$ and the $g_l$s must be contained in least two elements of $\overline{P^{\alpha, \beta}}$, and so $\text{cov}_{A}(P^{\alpha, \beta})\geq 2\alpha$.
	
	Let  $P^{\alpha, \beta}_{1}=\{e_1\}$ and $Q^{\alpha, \beta}_{1}=\{f_1, g_1 \ldots, g_\beta\}$, and for $2\leq i \leq \alpha$ let $P^{\alpha, \beta}_i=\{e_i\}$ and $Q^{\alpha, \beta}_{i}=\{f_i\}$.  Then $\{P^{\alpha, \beta}_1, Q^{\alpha, \beta}_1, \ldots, P^{\alpha, \beta}_\alpha, Q^{\alpha, \beta}_\alpha\}$ is a covering of $P^{\alpha, \beta}$ that satisfies both conditions A and B.  So $A_{ucg}(K_n, P^{\alpha, \beta})=2\alpha$ and $V(P^{\alpha, \beta})=2\alpha+\beta $.
	
\end{proof}

\section{Other Appendage Numbers}\label{sec:Gu}
In the paper \cite{Gu} Gu defines $A_{ucg}(C)$ to be the minimum number of vertices needed to be added to $C$ in order to create a uniform central graph $G$ with $\la\ZZ(G)\ra=C$.   To match notation  we let $A_{ucg}(C,-)=A_{ucg}(C)$. Gu's main theorem is the following.

\begin{thm}[Gu]  If $C$ is a connected graph, then
	\[A_{ucg}(C, - )=\begin{cases}
	2 & \text{if } C=\{v\}\\
	4  & \text{if } C=K_n, n\geq 2\\
	6 & \text{otherwise} \\
	\end{cases}\]	
\end{thm}
We use our results to give an alternative proof of Gu's result which is also true without the condition that $C$ is connected.

\begin{proof}
	Observe that $A_{ucg}(C, - )=\text{min}\{A_{ucg}(C, P)+|V(P)|\}$ where the minimum is taken over all graphs $P$.   By proposition \ref{r>1} $r(P)>1$ and so we may assume $|V(P)|\geq 2$. Since $\text{cov}_A(P)\geq 2$ for any $P$, by proposition \ref{prop:Knbound}  $A_{ucg}(K_n, P)\geq 2$ for $n\geq 2$, and by proposition \ref{prop:Cbound}  $A_{ucg}(C, P)\geq 4$ for any non-complete graph $C$.    Hence $A_{ucg}(K_n, -)\geq 4$ and $A_{ucg}(C, -)\geq 6$.
	
	Let  $P^2$ be a graph with two isolated vertices $\{u, v\}$  and  $\overline{P^2}=\{\{u\}, \{v\}\}$ a covering. Note $\overline{P^2}$ satisfies conditions A, B, A$'$ and B$'$.  Let \[G_1=\GG(K_n, P^2, \overline{P^2}, 1)-\{x_{0,1}\}\] \noindent and
	\[G_2=\GG(K_n, P^2, \overline{P_2}, 2)-\{x_{0,1}, x_{0,2}\}.\] \noindent The graph $G_1$ is the UCG in the first half of the proof of theorem \ref{thm:Knrange} and $G_2$ is a UCG by proposition  \ref{prop:A'B'UCG}.  Furthermore, $|\II(G_1)|=2$ and $|\II(G_1)|=4$, so $A_{ucg}(K_n, -)=4$ and  $A_{ucg}(C, -)=6$.

	Finally, $A_{ucg}(\{v\}, -)=2$ by theorem \ref{prop:Appendv}.
\end{proof}
	We can also consider $A_{ucg}(-, P)$,  the minimum number of vertices needed to be added to $P$ in order to construct a uniform central graph $G$ with $\la\CP(G)\ra=P$.  From propositions \ref{r>1} and \ref{prop:Appendv}, it follows that $A_{ucg}(-, P)=\infty$ if $r(P)=1$, and $A_{ucg}(-, P)=1$ otherwise.

\end{document}